\documentclass[11pt,letterpaper]{amsart}

  \usepackage{amssymb}
  \usepackage{amsmath}
  \usepackage{amsfonts}
  \usepackage{amscd}
  \usepackage{mathtools}
  \usepackage{xspace}
\usepackage[T1]{fontenc}
\usepackage[utf8]{inputenc}
\usepackage{float}
  \usepackage{tikz-cd}

  \usepackage{shuffle}
  \usepackage{hyperref}
  \usepackage{comment}
  \usepackage[symbol]{footmisc}

  \makeatletter
  \CheckCommand*\refstepcounter[1]{\stepcounter{#1}%
      \protected@edef\@currentlabel
       {\csname p@#1\endcsname\csname the#1\endcsname}%
  }
  \renewcommand*\refstepcounter[1]{\stepcounter{#1}%
    \protected@edef\@currentlabel
      {\csname p@#1\expandafter\endcsname\csname the#1\endcsname}%
  }
  \def\labelformat#1{\expandafter\def\csname p@#1\endcsname##1}
  \DeclareRobustCommand\Ref[1]{\protected@edef\@tempa{\ref{#1}}%
     \expandafter\MakeUppercase\@tempa
  }
  \makeatother

  \makeatletter
  \newcommand{\numberlike}[2]{%
     \expandafter\def\csname c@#1\endcsname{%
         \expandafter\csname c@#2\endcsname}%
  }
  \makeatother

  \def\DefaultNumberTheoremWithin{section}

  \theoremstyle{plain}
  \newtheorem{Lemma}{Lemma}
     \numberwithin{Lemma}{\DefaultNumberTheoremWithin}
     \labelformat{Lemma}{Lemma~#1}
  
     \numberwithin{Claim}{\DefaultNumberTheoremWithin}
     \numberlike{Claim}{Lemma}
     \labelformat{Claim}{Claim~#1}
  \newtheorem{Theorem}{Theorem}
     \numberwithin{Theorem}{\DefaultNumberTheoremWithin}
     \numberlike{Theorem}{Lemma}
     \labelformat{Theorem}{Theorem~#1}
  \newtheorem{Corollary}{Corollary}
     \numberwithin{Corollary}{\DefaultNumberTheoremWithin}
     \numberlike{Corollary}{Lemma}
     \labelformat{Corollary}{Corollary~#1}
  \newtheorem{Proposition}{Proposition}
     \numberwithin{Proposition}{\DefaultNumberTheoremWithin}
     \numberlike{Proposition}{Lemma}
     \labelformat{Proposition}{Proposition~#1}
  
     \numberwithin{Conjecture}{\DefaultNumberTheoremWithin}
     \numberlike{Conjecture}{Lemma}
     \labelformat{Conjecture}{Conjecture~#1}

  \theoremstyle{definition}
  
     \numberwithin{Definition}{\DefaultNumberTheoremWithin}
     \numberlike{Definition}{Lemma}
     \labelformat{Definition}{Definition~#1}

  \theoremstyle{definition}
  \newtheorem{Question}{Question}
     \numberwithin{Question}{\DefaultNumberTheoremWithin}
     \numberlike{Question}{Lemma}
     \labelformat{Question}{Question~#1}

  \theoremstyle{definition}
  
     \numberwithin{Problem}{\DefaultNumberTheoremWithin}
     \numberlike{Problem}{Lemma}
     \labelformat{Problem}{Problem~#1}

  \theoremstyle{remark}
  
     \numberwithin{Remark}{\DefaultNumberTheoremWithin}
     \numberlike{Remark}{Lemma}
     \labelformat{Remark}{Remark~#1}
  \theoremstyle{remark}

  \newtheorem{Example}{Example}
     \numberwithin{Example}{\DefaultNumberTheoremWithin}
     \numberlike{Example}{Lemma}
     \labelformat{Example}{Example~#1}
  
     \labelformat{Case}{Case~#1}
     \numberwithin{Case}{Lemma}
  
     \labelformat{Step}{Step~#1}
     \numberwithin{Step}{Lemma}

  \labelformat{equation}{(#1)}
  \labelformat{figure}{Figure~#1}
  \labelformat{section}{\S#1}
  \labelformat{subsection}{\S#1}
  \labelformat{subsubsection}{\S#1}

  \def\eqref{\ref}

  \newcommand{\defn}[1]{\textit{\textbf{#1}}}

  \allowdisplaybreaks

  \newcommand{\mb}{\mathbb}
  \newcommand{\mc}{\mathcal}

  \newcommand{\Zeta}{Z}
  \newcommand{\redH}{\widetilde{H}}
  \newcommand{\redC}{\widetilde{C}}
  \newcommand{\redM}{\widetilde{M}}
  \newcommand{\Hom}{H}

  \def\binomial(#1,#2){{#1\choose #2}}
\def\delplus(#1,#2,#3){\ensuremath{\pi}_{#1,#2,#3}}
\def\delminus(#1,#2){\ensuremath{\operatorname{\partial}}_{#1,#2}}
  
  \def\redword(#1){{\ensuremath{\tilde{#1}}}}

  \newcommand{\sgn}{\ensuremath{\operatorname{sgn}}}

  \newcommand{\JJJ}{\mc{J}}
  \newcommand{\LLL}{\mc{L}}
  \newcommand{\PPP}{\mc{P}}

  \newcommand{\KK}{\mb{K}}

  \newcommand{\ccf}{\mathfrak{c}}

  \newcommand{\lcm}{\mathrm{lcm}}

  \newcommand{\Sy}{\mathcal{S}}

  \newenvironment{note}[1][Note]
   {\bigskip\begin{center}\begin{boxedminipage}{4.5in}\setlength{\parindent}{1em}\noindent\textbf{#1. }}
   {\end{boxedminipage}\end{center}\bigskip}

\begin{document}

 \title[Subadditivity of shifts]{Subadditivity of shifts, Eilenberg--Zilber shuffle products and homology of lattices}


\date{\today}

  \author[K. Adiprasito]{Karim Adiprasito}
  \address{IMJ-PRG, CNRS, Sorbonne Université and Université Paris Cité, F-75005 Paris, France}
\email{karim.adiprasito@imj-prg.fr}

\author{Anders Bj\"orner}
     \address{Department of Mathematics \\
              Royal Institute of Technology (KTH) \\
              S-100 44 Stockholm \\
              Sweden}
     \email{bjoerner@math.kth.se}

 \author[J. Hakavuori]{Joel Hakavuori}
 \address{IMJ-PRG, Sorbonne Université, F-75005 Paris, France} 
 \email{hakavuori@imj-prg.fr}

 \author[M. Margaritis]{Minas Margaritis}
   \address{Department of Mathematics\\ National and Kapodistrian University of Athens\\ Panepistimioupolis\\ 15784 Athens\\ Greece}
   \email{minos2472002@gmail.com}

 \author[V. Welker]{Volkmar Welker}
    \address{Fachbereich Mathematik und Informatik\\             Philipps-Universit\"at Marburg\\
             35032 Marburg\\             Germany}
   \email{welker@mathematik.uni-marburg.de}

\begin{abstract}
We show that the maximal shifts in the minimal free resolution of 
the quotients of a polynomial ring by a monomial ideal are subadditive as a function of the homological degree. This answers a question that has received some attention in recent years. To do so, we define and study a new model for the homology of posets, given by the so-called synor complex. We also introduce an Eilenberg--Zilber-type shuffle product on the simplicial chain complex of lattices.
Combining these concepts we prove that the existence of a nonzero homology class for a lattice forces certain nonzero homology classes in lower intervals. This result then translates into properties of the minimal free resolution. In particular, it yields a strengthening of the original subadditivity statement.
	\end{abstract}
\maketitle

\section{Introduction}
\label{sec:intro} 

  Let $\mathfrak{I}$ be a homogeneous ideal in a 
  standard graded $\KK$-algebra $S$ for some field $\KK$ and $\beta_{ij}(S/\mathfrak{I}) := \dim_\KK \mathrm{Tor}_i^S(S/\mathfrak{I},\KK)_j$ the \defn{graded Betti numbers} of $S/\mathfrak{I}$.
   The maximal shift
   $t_i(S/\mathfrak{I})$ in the $i$th homological degree is the maximal 
   $j$ such that
   $\beta_{ij}(S/\mathfrak{I}) \neq 0$. If $\beta_{ij}(S/\mathfrak{I}) = 0$ for all $j$ we set $t_i(S/\mathfrak{I}) = 0$.
   The following question has been considered (see \cite{Avramov, EHU}):
   

   \medskip

   \begin{Question}
     \label{sub}
     For which classes of homogeneous ideals $\mathfrak{I}$
     does 
     $$t_{i_1+i_2}(S/\mathfrak{I}) \leq t_{i_1}(S/\mathfrak{I}) + t_{i_2}
     (S/\mathfrak{I})$$
     hold for all $0 \leq i_1, i_2$.
   \end{Question}


   \medskip

   Since $t_0(S/\mathfrak{I}) = 0$ the question is interesting only for 
   $i_1,i_2 \geq 1$.
   It has a negative answer for general homogeneous ideals
   even when $S$ is a polynomial ring, and counterexamples can be found in \cite{EHU} and \cite{Avramov}.
   
   For the rest of this paper $S = \KK[x_1,\ldots, x_n]$ and 
   $\mathfrak{I}$ is a monomial ideal. Indeed, as a
   special case of 
   \cite[Conjecture 6.4]{Avramov} it was conjectured that
   \ref{sub} holds for quadratic monomial ideals. In this case,
   their equation \cite[(6.4)]{Avramov} implies that \eqref{sub}
   holds if $\mathrm{char}(\KK) = 0$ and $1$ is added on the right-hand side.
   Over the last 10 years numerous results were obtained providing
   positive evidence that \ref{sub} has a positive answer for all monomial ideals.
   In \cite{HerzogSrinivasan} the case $i_1 = 1$ was settled, with an extension appearing in \cite{AA}. Positive results for 
   edge ideals and other graph-motivated situations can be found in \cite{AbNev, BH, JayKumar} and facet ideals of simplicial forests were treated in \cite{faridi}. The case when the monomial ideal has a DGA resolution was settled in \cite{katthaen}.
   

   For a monomial ideal $\mathfrak{I}$ with minimal monomial generating set $G(\mathfrak{I})$ the \defn{LCM-lattice} $\LLL(\mathfrak{I})$ is the
   set of all least common multiples $\displaystyle{\lcm_{m \in B}} m$ for $B \subseteq G(\mathfrak{I})$ where we set $\displaystyle{\lcm_{m \in \emptyset}} m = 1$.
   Ordered by divisibility the set $\LLL(\mathfrak{I})$ becomes an atomic lattice. 
   
   Our approach to answering \ref{sub} leads to the study of LCM-lattices and their topology. 
   For the latter, when referring to the topology or homology of a poset $\PPP$, we mean the topology or homology of its order complex; that is the simplicial complex of totally ordered subsets of $\PPP$.

   For $p \leq q$ in $\PPP$ we denote by $[p,q]$ the \defn{closed interval} 
   $\{ x \in \PPP ~|~p \leq x \leq q\}$. \defn{Open} and
   \defn{half-open intervals} in $\PPP$ are defined analogously. 

   Our work on \ref{sub} is based on the known relation (see \cite{GPW}) between graded Betti numbers $\beta_{ij}(S/\mathfrak{I})$ and homology groups of open intervals supported in $\LLL=\LLL(\mathfrak{I})$,
   given by

   \begin{align} \label{eq:lcm}
   \beta_{ij}(S/\mathfrak{I}) & = \sum_{\genfrac{}{}{0pt}{}{m \in \LLL(\mathfrak{I})}{\deg(m) = j}} \dim_\KK \redH_{i-2} \big( \,(1,m)\,\big),
   \end{align}
   
   where the homology is reduced and taken with $\KK$-coefficients. We adopt this convention for the rest of the paper.
   
   Thus if $1 \leq i_1,i_2$ and $t_{i_1+i_2} > 0$ then $\redH_{i_1+i_2-2} \big(\,(1,m)\,\big) \neq 0$
   for some monomial $m$ of degree $t_{i_1+i_2}$.
   This shows that \ref{sub} has a positive answer when in this situation
   there are monomials $n_1$ and $n_2$ in our lattice such that $\redH_{i_1-2} \big(\,(1,n_{1})\,\big) \neq 0$ and $\redH_{i_2-2} \big(\,(1,n_{2})\,\big) \neq 0$
   and $t_{i_1+i_2}$ is bounded from above by the degree of the lcm of $n_1$ and $n_2$. Hence, the question for monomial ideals is implied by the following theorem. In its formulation we write $\hat{0}$ for the unique
   minimal element of $\LLL$ and write $n_1 \vee n_2$ for
   the supremum of two elements of $\LLL$. We prove:
   
   \begin{Theorem}
   \label{mainstart}
      Let $1 \leq i_1,i_2$ and let $\LLL$ be a finite lattice with an element $m$ such that $$\redH_{i_1+i_2-2}\big(\,(\hat{0},m)\,\big) \neq 0.$$ Then there exist elements
      $n_1,n_2 \in \LLL$ such that $$\redH_{i_1-2} \big(\,(\hat{0},n_{1})\,\big) \neq 0, \ \redH_{i_2-2} \big(\,(\hat{0},n_{2})\,\big) \neq 0$$ 
      and $n_1\vee n_2 = m$.
   \end{Theorem}
Let us note that this statement is immediate for several classes of lattices, such as face lattices of strongly regular CW complexes. 

As a consequence of \ref{mainstart} and the arguments preceding it we obtain a positive answer to \ref{sub} for monomial ideals.

\begin{Theorem} \label{thm:subad}
   Let $\mathfrak{I}$ be a monomial ideal in $S$. 
   Then for all $0 \leq i_1,i_2$ we have
     $$t_{i_1+i_2}(S/\mathfrak{I}) \leq t_{i_1}(S/\mathfrak{I}) + t_{i_2}
     (S/\mathfrak{I}).$$
\end{Theorem}

For the proof of \ref{mainstart} we proceed as follows.
In \ref{sec:lemmas} we introduce a shuffle product on the
chain complex of the order complex of a finite lattice.  
Then in \ref{sec:synor} we establish the existence of a
subcomplex of the chain complex of a poset - a so-called synor complex - 
which encodes the homology of all lower intervals. The synor complex is
graded by the synors of the poset, which are the elements for which the
subposet of elements below is not acyclic. 
Afterwards in \ref{sec:shuffling} we leverage the 
shuffle product and
the synor complex to prove in \ref{lmuff} a result on the representation of cycles as sums of shuffles of synor chains. As a consequence we can prove 
\ref{mainthm} which contains \ref{mainstart} as a special case.
In \ref{sec:further} we apply \ref{mainthm} to 
prove in \ref{thm:ksubad} a generalization of \ref{thm:subad} as well as a bound
on the number of distinct multigraded shifts in the minimal free resolution
of a monomial ideal. We also show in \ref{thm:synorres} 
that the synor complex
from \ref{sec:synor} can be used to construct a minimal free
resolution of a monomial ideal.

\subsection*{Acknowledgements:} We thank Aldo Conca, David Eisenbud, Craig Hu\-neke and Eran Nevo for their interest and helpful comments. This paper was written while the first author was a member at the Institute for Advanced Study in Princeton. 
Funding for his membership has been provided by The Ambrose Monell Foundation and Horizon Europe ERC Grant number: 101045750 / Project acronym: HodgeGeoComb. The third and fourth authors would like to thank the Einstein Institute of Mathematics at the Hebrew University of Jerusalem for the hospitality.
   
\section{Posets, homology and the shuffle product in lattices}
         \label{sec:lemmas}

We first introduce the concepts for posets and lattices we use in this 
paper.

For a poset $\PPP$ we call a multichain $\ccf = (c_0 \geq c_{1} \geq \cdots \geq c_k)$ of elements
of $\PPP$ an \defn{order multichain}. We call $k$ the \defn{length} of the
multichain. Order multichains of length $k$ are simply referred to as \defn{$k$-multichains}. We call an order multichain with 
no repetitions an \defn{order chain}.  
Note that the empty multichain $\ccf =()$ is the
unique order multichain of length $-1$.
The \defn{order complex} of $\PPP$ is the simplicial complex of all
order chains in $\PPP$. 
From now on 

\smallskip

\centerline{all posets and lattices are assumed to be finite.}

\smallskip

We adopt the notions of open, half-open and closed intervals in posets,
as defined in \ref{sec:intro}. We call a subset $\mathcal{J} \subseteq \PPP$
an \defn{order ideal} in $\PPP$ if for $x,y \in \PPP$, $x \leq y$ and $y \in \JJJ$ imply $x \in \JJJ$. For $x \in \PPP$ we write
$\PPP_{\leq x}$ for the order ideal of all $y \leq x$ and 
$\PPP_{< x}$ for the order ideal $\PPP_{\leq x} \setminus \{x\}$.

Next we recall some basics about the homology of order complexes of a poset
$\PPP$.
We fix a field $\KK$ and for $k\geq -1$, we denote by $\redM_{k}(\PPP)$ the $\KK$-vector space freely 
generated by the order multichains of length $k$. We write $\redC_k(\PPP)$ for the subspace spanned by the order chains of length $k$. 
For an order multichain $\ccf = (c_0 \geq c_{1} \geq \cdots \geq c_k)$
of length $k$ and $0 \leq j \leq k$ we set $$\partial_k^ {(j)}(\ccf)
= (c_0 \geq \cdots \geq c_{j-1} \geq c_{j+1} \geq \cdots \geq c_k)$$
and write $$\partial_k = \sum_{j=0}^ k  (-1)^j \partial_k^{(j)}$$ for the induced
linear map $\partial_k : \redM_k(\PPP) \rightarrow \redM_{k-1}(\PPP)$. 
Thus $(\redM_*,\partial_*)$ is the standard simplicial chain complex 
of the simplicial set associated to the order complex of $\PPP$, i.e., the simplicial complex of multichains in $\PPP$. 
The complex 
$(\redC_*,\partial_*)$ is the standard simplicial chain complex associated to the order complex of $\PPP$. The order complex of any subposet $\mathcal{Q} \subseteq \PPP$ is naturally a subcomplex of the order complex of $\PPP$. This gives a simplicial pair $(\PPP, \mathcal{Q})$, for which $(\redC_* (\PPP, \mathcal{Q}), \partial_*)$ denotes the relative simplicial chain complex and $\redH_* (\PPP, \mathcal{Q})$ the homology of $(\redC_* (\PPP, \mathcal{Q}), \partial_*)$. 

By \cite{EM} we have that 
$(\redC(\PPP),\partial_*)$ is the normalization of $(\redM_*(\PPP),\partial_*)$ 
and in particular by the proof of \cite[Theorem 4.1]{EM} we have:

\begin{Lemma} \label{lem:norm}
  The projection map $\pi:\redM_*(\PPP) \rightarrow \redC_*(\PPP)$ which acts identically on order chains and vanishes on multichains that are not order chains, induces a homotopy inverse to the
  inclusion $(\redC_*(\PPP),\partial_*) \hookrightarrow (\redM_*(\PPP),\partial_*)$.
\end {Lemma}

Using this lemma we will work with $\redC_*(\PPP)$ in the following
sections \ref{sec:shuffling} and \ref{sec:further} and ignore multichains which appear in any of the calculations.
For this section, \ref{sec:lemmas}, we mostly work in $(\redM_*(\PPP),\partial_*)$ since this
facilitates the manipulation of the combinatorial product on
$(\redM_*(\PPP),\partial_*)$ in case $\PPP$ is a lattice, 
which we define below.

Before we can define this product, we set up notation which will be used in the later
section for all posets.
We will consider an $i$-multichain $\ccf$ in $\PPP$ as an element
of the $(i+1)$-fold Cartesian product $\PPP^{i+1}$. For an arbitrary $(i+1)$-tuple $\ccf$ in $\PPP^ {i+1}$ and an arbitrary $(j+1)$-tuple $\ccf' \in \PPP^{j+1}$
we write $\ccf \cdot \ccf'$ for the $(i+j+2)$-tuple in $\PPP^{i+j+2}$ 
that is the concatenation of $\ccf$ and $\ccf'$.
In case $\ccf \cdot \ccf'$ is an order multichain, i.e., both $\ccf$ and $\ccf'$ are order multichains and the minimal element of
$\ccf$ is greater than or equal to the maximal element of $\ccf'$, 
then we write $\ccf * \ccf'$ for this order multichain. 
We write 
$\min(\ccf)$ for the minimal element of an order multichain $\ccf$.
If $\ccf$ is an $i$-multichain, $x = \min(\ccf)$ and
$\gamma = \sum_{\ell=1}^k \lambda_\ell \, \ccf_\ell \in \redM_j(\PPP)$ where
$\ccf_\ell$ are order multichains supported in $\PPP_{\leq x}$ then
we write $\ccf * \gamma$ for the chain 
$\sum_{\ell=1}^k \lambda_\ell \, \ccf*\ccf_\ell$ in $\redM_{i+j+1}(\PPP)$.


\medskip

Let $\LLL$ be a lattice.
Since by our assumption $\LLL$ is finite, it
has a unique minimal element which we denote by $\hat{0}$ and
a unique maximal element which we denote by $\hat{1}$. 
Note that if $\LLL$ is the lcm-lattice $\lcm(\mathfrak{I})$ of a monomial
ideal $\mathfrak{I}$ then $\hat{0} =1$, where $1$ stands for the monomial
$x_1^0 \cdots x_n^ 0$. For $x,y \in \LLL$ we write $x \vee y$ for the
\defn{supremum} or \defn{join} of $x$ and $y$, and $x \wedge y$ for the 
\defn{infimum} or \defn{meet} of $x$ and $y$. In $\lcm(\mathfrak{I})$ the
join of two elements is simply their lcm. 

We define $\tau : \LLL^{k+1} \rightarrow \widetilde{M}_k(\LLL)$ as 
the operator which takes the $(k+1)$-tuple 
$(a_0 ,\ldots, a_k) \in \LLL^{k+1}$ to the $k$-multichain
$(a_0\vee \cdots \vee a_{k} \geq a_1 \vee \cdots \vee a_{k} \geq \cdots \geq a_k)$. 

For a permutation $\sigma\in S_{k+1}$ of $\{0,\ldots, k\}$, a set $X$ and a $(k+1)$-tuple
$a = (a_0 ,\ldots, a_k) \in X^{k+1}$ we write $a_\sigma$ or $(a_0 ,\ldots, a_k)_{\sigma}$ for the $(k+1)$-tuple
$(a_{\sigma(0)},\ldots, a_{\sigma(k)})$. We will usually take $X$ to be a lattice $\LLL$ or the set of natural numbers.  

We say a permutation $\sigma\in S_{i+j+2}$ viewed as a bijection $\sigma:\{0,1,\ldots, i+j+1\}\rightarrow \{0,1,\ldots, i+j+1\}$ is
an \defn{$(i,j)$-shuffle} if $\sigma^{-1}$ is increasing on the sets $\{0,1,\ldots, i\}$ and $\{i+1,i+2,\ldots, i+j+1\}$.
We write $S_{i,j}$ for the set of all $(i,j)$-shuffles inside $S_{i+j+2}$. 

The shuffle operator 
$\shuffle : \widetilde{M}_i(\LLL) \times \widetilde{M}_j(\LLL) 
\rightarrow \widetilde{M}_{i+j+1}(\LLL)$ is then defined as the linear extension
of the map sending an $i$-multichain $\ccf$ and a $j$-multichain $\ccf'$ to

$$\ccf \shuffle \ccf' := \sum_{\sigma \in S_{i,j}} \sgn(\sigma) \,\tau\big( \,(\ccf \cdot \ccf')_\sigma\,\big).$$

Note that if $i=-1$ then $\ccf \shuffle \ccf' = \ccf'$. 
Consider the second tensor power 
$$(D_*(\LLL),\delta_*) = (\widetilde{M}_*,\partial_*) \otimes
(\widetilde{M}_*,\partial_*)$$ 
of the chain complex $(\widetilde{M}_*,\partial_*)$, 
with differential $\delta_n = \bigoplus_{i+j+1=n} \big(\,(\partial_i,\text{id})+ (-1)^{i+1} (\text{id},\partial_j)\,\big)$.

Since $\widetilde{M}_i (\LLL) \otimes \widetilde{M}_j(\LLL)$ has a basis 
consisting of the elementary tensors $\ccf \otimes \ccf'$ 
for $i$-multichains $\ccf$ and $j$-multichains $\ccf'$, it 
follows that $\ccf \otimes \ccf' \mapsto \ccf \shuffle \ccf'$
induces a map of $\mathbb{K}$-vector spaces  
$\shuffle : D_n(\LLL) \rightarrow \widetilde{M}_{n} (\LLL)$. In \ref{prop:shuffle} we prove that this map is a chain map. This will yield the very helpful boundary formula in \ref{lem:shuffle0}.
Before that, we prove two technical but straightforward combinatorial lemmas.\medskip 

Let $X^{i,j}=\{0,\ldots,i+j+1\}\times S_{i,j}$, and consider the partition  $X^{i,j}=\cup_{s=1}^{4}A^{i,j}_{s}$, where \begin{eqnarray*} 
     A^{i,j}_{1} & = & \Big\{\,(\ell,\sigma)\in X^{i,j}\,:\genfrac{}{}{0pt}{}{\ell\geq 1, \sigma(\ell-1)\in \{0,\ldots,i\} \ \text{and}}{ \sigma(\ell)\in \{i+1,\ldots,i+j+1\}}\,\Big\} \\
      A^{i,j}_{2} & = & \Big\{\,(\ell,\sigma)\in X^{i,j}\,:\genfrac{}{}{0pt}{}{\ell \geq 1, \sigma(\ell-1)\in \{i+1,\ldots,i+j+1\}\ \text{and}}  {\sigma(\ell)\in \{0,\ldots,i\}}\,\Big\} \\
      A^{i,j}_{3}& = & \big\{\,(\ell,\sigma)\in X^{i,j}\,:\, \sigma(\ell-1),\sigma(\ell)\in \{0,\ldots,i\}\,\big\} \\
      A^{i,j}_{4}& = & \big\{\,(\ell,\sigma)\in X^{i,j}\,:\, \sigma(\ell-1),\sigma(\ell)\in \{i+1,\ldots,i+j+1\}\,\big\} 
  \end{eqnarray*} where  we interpret the conditions $\sigma(-1)\in \{0,\ldots,i\}$ and $\sigma(-1)\in \{i+1,\ldots,i+j+1\}$ as true.
\begin{Lemma}\label{Tbij} Consider the map $T_{i,j}:A^{i,j}_{1}\rightarrow A^{i,j}_{2}$ defined by $$ T_{i,j}(\ell,\sigma)=(\ell,\sigma\circ \tau_{\ell}),$$ where $\tau_{\ell}\in S_{i+j+2}$ denotes the transposition $(\ell-1,\ell)$. Then, 
\begin{itemize}
    \item[(i)] $T_{i,j}$ is a bijection.
    \item[(ii)] For every $\sigma\in S_{i+j+2}$ we have $\sgn(\sigma)=-\sgn(\sigma\circ\tau_{\ell})$
    \item[(iii)] For every $\ccf\in\LLL^{i+1},\ccf'\in\LLL^{j+1}$ and $(\ell,\sigma)\in A^{i,j}_{1}$ we have $$\partial_{i+j+1}^{(\ell)}\left[\tau\big(\,(\ccf \cdot \ccf' )_{\sigma}\,\big)\right]=\partial_{i+j+1}^{(\ell)}\left[\tau\big(\,(\ccf \cdot \ccf' )_{\sigma\circ\tau_{\ell}}\,\big)\right]$$
\end{itemize}
\end{Lemma}
\begin{proof} The first two properties follow immediately from the definitions, so we only prove (iii). 

Let $\ccf  = (c_0 \geq \cdots \geq c_i)$ 
  and $\ccf' = (c_{i+1} \geq \cdots \geq c_{i+j+1})$ viewed as tuples. 
  For a $\sigma \in S_{i,j}$ we write $\sigma_1(k)$ for 
  $\min\{ \{0,\ldots, i\} \cap \{ \sigma(k),\ldots, \sigma(i+j+1) \}\}$
  and $\sigma_2(k)$ for 
  $\min\{ \{i+1,\ldots, i+j+1\} \cap \{ \sigma(k),\ldots, \sigma(i+j+1) \}\}$
  where we consider the minimum over the empty set as $-\infty$. 
  
  Setting $c_{-\infty} = c'_{-\infty} = \hat{0}$ we get
  that 
  $$\tau\big(\,(\ccf \cdot \ccf')_\sigma\,\big) = 
  (c_{\sigma_1(0)} \vee c'_{\sigma_2(0)}\geq \cdots \geq c_{\sigma_1(i+j+1)} \vee
   c'_{\sigma_2(i+j+1)}).$$ Assuming $(\ell,\sigma)\in A^{i,j}_{1}$, we note that the sets $\{\sigma(k),\ldots,\sigma(i+j+1)\}$ and $\{(\sigma\circ\tau_{\ell})(k),\ldots,(\sigma\circ\tau_{\ell})(i+j+1)\}$ coincide for all $k\neq \ell$. This is because $(\sigma\circ\tau_{\ell})(j)=\sigma(j)$ for $j\neq \ell-1,\ell$, while $(\sigma\circ\tau_{\ell})(\ell-1)=\sigma(\ell)$ and 
 $(\sigma\circ\tau_{\ell})(\ell)=\sigma(\ell-1)$. Consequently, $\sigma_{1}(k)=(\sigma\circ\tau_{\ell})_{1}(k)$ and $\sigma_{2}(k)=(\sigma\circ\tau_{\ell})_{2}(k)$ for all $k\neq \ell$.
  This means that $\tau\big(\,(\ccf \cdot \ccf')_{\sigma\circ\tau_{\ell}}\,\big)$ coincides with $\tau\big(\,(\ccf \cdot \ccf')_\sigma\,\big)$, except for the entry corresponding to $\ell$. Removing this entry implies the equality in (iii).
\end{proof}
 \begin{Lemma}\label{Omegas}The map $\Omega_{i,j}:A^{i,j}_{3}\rightarrow
 \{0,\ldots,i\}\times S_{i-1,j}$ defined by 
 \begin{equation}\label{Ome}\Omega_{i,j}(\ell,\sigma)=(m,\phi)\in \{0,\ldots,i\}\times S_{i-1,j} \ \text{\emph{iff\footnotemark[3]}} \ \phi\circ\partial_{i+j+1}^{(m)}=\partial_{i+j+1}^{(\ell)}\circ \sigma
 \end{equation} 
 \footnotetext[3]{Here, $\partial_{i+j+1}^{(k)}$ acts on an $(i+j+2)$-tuple by removing its $k$-th element from the left, for $k=0,\ldots,i+j+1$. Permutations act on tuples as usual.}
 has the following properties:
 \begin{itemize}
 \item[(i)] $\Omega_{i,j}$ is a bijection.
 \item[(ii)] For every $(\ell,\sigma)\in A^{i,j}_{3}$, we have that $(-1)^{\ell}\sgn(\sigma)=(-1)^{m}\sgn(\phi)$, where $(m,\phi)=\Omega_{i,j}(\ell,\sigma)$.
 \item[(iii)] For every $\ccf\in\LLL^{i+1},\ccf'\in\LLL^{j+1}$ and $(\ell,\sigma)\in A^{i,j}_{3}$ we have that $$\partial_{i+j+1}^{(\ell)}\left[\tau \big(\,(\ccf \cdot \ccf')_{\sigma}\,\big)\right]=\tau \big(\partial_{i+j+1}^{(\ell)}\left[(\ccf \cdot \ccf')_{\sigma}\right]\big)=\tau \big((\partial_{i+j+1}^{(m)}\left[\ccf \cdot \ccf'\right])_{\phi}\big)$$ where $(m,\phi)=\Omega_{i,j}(\ell,\sigma)$.
 
 \end{itemize}
 \end{Lemma}
 \begin{proof} 

  \noindent \item[(i)] 
  To prove that $\Omega_{i,j}$ is a well defined bijection, we first evaluate the condition in \ref{Ome} at the tuple $\{0,\ldots,i+j+1\}$. The equivalent condition is thus \begin{equation}\label{phi}(0,\ldots,\widehat{m},\ldots,i+j+1)_{\phi}=(\sigma(0),\ldots,\widehat{\sigma(\ell)},\ldots,\sigma(i+j+1)).\end{equation} Note that the left hand side describes the image of the permutation $\phi \in S_{i+j+1}$ acting on the (ordered) $(i+j+1)$-tuple $(0,\ldots,\hat{m},\ldots,i+j+1)$. This image is clearly not the same as $(\phi(0),\ldots,\widehat{\phi(m)},\ldots,\phi(i+j+1))$ due to the jump created by the absence of $m$.

  Given a pair $(\ell,\sigma)\in A^{i,j}_{3}$, we can immediately identify $m=\sigma(\ell)$ by comparing missing elements in the two sides of \ref{phi}. Then, we must have $m\in\{0,\ldots,i\}$ due to $(\ell,\sigma)\in A^{i,j}_{3}$. To determine $\phi\in S_{i+j+1}$, we identify its image $(0,\ldots,i+j)_{\phi}=(\phi(0),\phi(1),\ldots,\phi(i+j))$. By \ref{phi}, this image can be (uniquely) constructed from $(\sigma(0),\ldots,\widehat{\sigma(\ell)},\ldots,\sigma(i+j+1))$ if we decrease by $1$ all entries $\sigma(k)$ with $\sigma(k)>\sigma(\ell)$. Comparing the images of $\sigma$ and $\phi$ through this mechanism, it is clear that $\sigma\in S_{i,j}$ implies $\phi\in S_{i-1,j}$.

  Conversely, assume we are given a pair $(m,\phi)\in \{0,\ldots,i\}\times S_{i-1,j}$. In order to recover the unique pair $(\ell,\sigma)\in A^{i,j}_{3}$ satisfying \ref{phi}, we again begin by identifying missing elements, so that $\sigma(\ell)=m$. We then use the previous mechanism in reverse: Start with $(\phi(0),\phi(1),\ldots,\phi(i+j))$, which is known, and increase all $\phi(k)$ with $\phi(k)\geq m$ by $1$. This is precisely $(0,\ldots,\widehat{m},\ldots,i+j+1)_{\phi}$ which is $(\sigma(0),\ldots,\widehat{\sigma(\ell)},\ldots,\sigma(i+j+1))$ by \ref{phi}, provided we can solve for $(\ell,\sigma)\in A^{i,j}_{3}$. 
  
  We almost have the full image of $\sigma$ figured out - we only need to determine the value of $\ell$, because then we can place $m=\sigma(\ell)$ back in the entry indexed by $\ell$. This is where the condition $(\ell,\sigma)\in A^{i,j}_{3}$ comes into play: The fact that $\sigma(\ell-1),\sigma(\ell)\in \{0,\ldots,i\}$ along with the shuffle condition $\sigma \in S_{i,j}$ yield that $\sigma(\ell-1)=\sigma(\ell)-1=m-1$ in case $\ell \neq 0$. Looking back into \ref{phi}, this says that $\ell$  should be one larger than the position of $m-1$ in $(0,\ldots,\widehat{m},\ldots,i+j+1)_{\phi}$ provided that $m\neq 0$, and it should be equal to $0$ otherwise.
  It is then easy to check that the resulting pair $(\ell,\sigma)$ satisfies $(\ell,\sigma)\in A^{i,j}_{3}$ and $\Omega_{i,j}(\ell,\sigma)=(m,\phi)$.
     
     \noindent (ii)    
     Let $(\ell,\sigma)\in A^{i,j}_{3}$ and $\Omega_{i,j}(\ell,\sigma)=(m,\phi)\in \{0,\ldots,i\}\times S_{i-1,j}$. We look at the tuples in \ref{phi} and place the element $m=\sigma(\ell)$ in the leftmost position of both sides, thus creating (the image of) a permutation $\pi\in S_{i+j+2}$ written in two ways. The stated equality then follows if we compute the sign of $\pi$ using the two expressions: 
     
     For the left hand side we compute the number of inversions, which, compared to the number of inversions of $\phi$, has increased by $m$. Consequently, \begin{equation}\label{sgnphi}\sgn(\pi)=(-1)^{m}\sgn(\phi).\end{equation}

     For the right hand side, we note that we need $\ell$ transpositions involving $m$ in order to reach the permutation $\sigma$ from $\pi$. Thus, 
     \begin{equation}\label{sgnsigma}\sgn(\pi)=(-1)^{\ell}\sgn(\sigma).\end{equation}

      Comparing \ref{sgnphi} and \ref{sgnsigma}, we arrive at the desired equality.
     
     \noindent (iii)    We use the same notation as in the proof of \ref{Tbij}: For $(\ell,\sigma)\in A^{i,j}_{3}$, $\ccf\in \LLL^{i+1}$ and $\ccf'\in \LLL^{j+1}$ we write $$\tau\big(\,(\ccf \cdot \ccf')_\sigma\,\big) = 
  (c_{\sigma_1(0)} \vee c'_{\sigma_2(0)}\geq \cdots \geq c_{\sigma_1(i+j+1)} \vee
   c'_{\sigma_2(i+j+1)}).$$ We first observe that the element $\sigma(\ell)$ plays no role in the calculation of $\sigma_{2}(k)=\min\{ \{i+1,\ldots, i+j+1\} \cap \{ \sigma(k),\ldots, \sigma(i+j+1) \}\}$, since $(\ell,\sigma)\in A^{i,j}_{3}$ implies $\sigma(\ell)\in \{0,\ldots,i\}$. For $k>\ell$, the same clearly holds for $\sigma_{1}(k)$, since the sets $\{\sigma(k),\ldots,\sigma(i+j+1)\}$ do not contain $\sigma(\ell)$. 
   
   For $k<\ell$, the element $\sigma(\ell)$ is again irrelevant to the calculation of $\sigma_{1}(k)$:
   To see this, note that $(\ell,\sigma) \in A^{i,j}_{3}$ implies $\sigma(\ell-1),\sigma(\ell)\in \{0,\ldots,i\}$ and also $\sigma\in S_{i,j}$. Thus, $\sigma(\ell-1)<\sigma(\ell)$ by the shuffle condition, which means $\sigma_{1}(k)\leq\sigma(\ell-1)$. 
   
   All in all, the effect of $\sigma(\ell)$ appears only on the element of the multichain $\tau\big(\,(\ccf \cdot \ccf')_\sigma\,\big)$ corresponding to the index $\ell$. Hence,  $\tau\big(\,\partial_{i+j+1}^{(\ell)}\left[(\ccf \cdot \ccf')_\sigma\right]\,\big)$ is the same as $\tau\big(\,(\ccf \cdot \ccf')_\sigma\,\big)$ with the element corresponding to the index $\ell$ removed. This is precisely the first equality in (iii). The second equality follows directly from the definition of $\Omega_{i,j}$ in \ref{Ome}.
 \end{proof}
 We now go on to prove the main result of this section:
\begin{Proposition}
   \label{prop:shuffle}
   For $n \geq 0$ and $\alpha \in D_n(\LLL)$ we have
   $$\partial_{n} ( \shuffle ( \alpha)) = \shuffle (\delta_n(\alpha)).$$
\end{Proposition}
\begin{proof}
  It suffices to check the identity for 
  $\alpha = \ccf \otimes \ccf'$ where $\ccf$ is an $i$-multichain,
  $\ccf'$ a $j$-multichain and $i+j+1=n$. 
  Let $\ccf  = (c_0 \geq \cdots \geq c_i)$ 
  and $\ccf' = (c_{i+1} \geq \cdots \geq c_{i+j+1})$.

  Then, 
  \begin{equation*} 
    \ccf \shuffle \ccf'  =  \sum_{\sigma \in S_{i,j}} \sgn(\sigma)  
       \,\tau\big(\, (c\cdot c')_\sigma\,\big)
  \end{equation*}
  and
  \begin{equation*} 
    \label{eq:a} \partial_{i+j+1}(\ccf \shuffle \ccf')  =  \sum_{(\ell,\sigma)\in X^{i,j}}(-1)^{\ell} \sgn(\sigma) 
  \,\partial_{i+j+1}^{(\ell)}\left[\tau\big(\,(\ccf\cdot\ccf')_{\sigma}\,\big)\right]
  \end{equation*} 
  To make the notation cleaner, we set $\partial_{i+j+1}^{(\ell)}\left[\tau\big(\,(\ccf\cdot\ccf')_{\sigma}\,\big)\right]=g_{\ell}(\ccf,\ccf',\sigma)$.
We may decompose this sum as
  \begin{equation}\label{decomp}\sum_{(\ell,\sigma)\in X^{i,j}}(-1)^{\ell} \sgn(\sigma) 
  g_{\ell}(\ccf,\ccf',\sigma)=\sum_{s=1}^{4}\sum_{(\ell,\sigma)\in A^{i,j}_{s}}(-1)^{\ell} \sgn(\sigma) 
  g_{\ell}(\ccf,\ccf',\sigma).
  \end{equation}

We now analyze the terms appearing on the right hand side of \ref{decomp}
for $s=1,2,3,4$:

  \noindent ($s=1,2$) The term corresponding to $A^{i,j}_{1}\cup A^{i,j}_{2}$ vanishes, meaning that $$\sum_{s=1}^{2}\sum_{(\ell,\sigma)\in A^{i,j}_{s}}(-1)^{\ell} \sgn(\sigma) 
  g_{\ell}(\ccf,\ccf',\sigma)=0.$$
  
  Indeed, by \ref{Tbij} we have $g_{\ell}(\ccf,\ccf',\sigma\circ\tau_{\ell})=g_{\ell}(\ccf,\ccf',\sigma)$ and $\sgn(\sigma)=-\sgn(\sigma\circ\tau_{\ell})$ for $\tau_{\ell}$ the transposition appearing in the definition of the bijection $T_{i,j}$. Thus, \begin{align*}&\sum_{s=1}^{2}\sum_{(\ell,\sigma)\in A^{i,j}_{s}} (-1)^{\ell}\sgn(\sigma)g_{\ell}(\ccf,\ccf',\sigma) =\\
      & = \sum_{(\ell,\sigma)\in A^{i,j}_{1}} (-1)^{\ell}\big(\sgn(\sigma) 
  g_{\ell}(\ccf,\ccf',\sigma)+\sgn(\sigma\circ \tau_{\ell}) 
  g_{\ell}(\ccf,\ccf',\sigma\circ\tau_{\ell})\big)\\ &=\sum_{(\ell,\sigma)\in A^{i,j}_{1}}(-1)^{\ell} \big( \sgn(\sigma) 
  g_{\ell}(\ccf,\ccf',\sigma)-\sgn(\sigma) 
  g_{\ell}(\ccf,\ccf',\sigma)\big)=0.\end{align*}

  \noindent ($s=3$) The term corresponding to $A^{i,j}_{3}$ coincides with $\partial_i(\ccf) \shuffle \ccf' = \shuffle(\partial_i(\ccf) \otimes \ccf')$. 

      By property (iii) of \ref{Omegas} we can write \begin{equation*}\sum_{(\ell,\sigma)\in A^{i,j}_{3}}(-1)^{\ell} \sgn(\sigma) 
  g_{\ell}(\ccf,\ccf',\sigma)=\sum_{\genfrac{}{}{0pt}{}{(\ell,\sigma) \in A^{i,j}_{3}}{(m,\phi)=\Omega(\ell,\sigma)}}(-1)^{\ell} \sgn(\sigma) 
  \tau \big((\partial_{i+j+1}^{(m)}[\ccf \cdot \ccf'])_{\phi}\big)\end{equation*} and then by properties (i), (ii) of the same lemma this equals
  \begin{align*}
      & \sum_{\genfrac{}{}{0pt}{}{(\ell,\sigma) \in A^{i,j}_{3}}{(m,\phi)=\Omega(\ell,\sigma)}}(-1)^{\ell} \sgn(\sigma) 
  \tau \big((\partial_{i}^{(m)}(\ccf) \cdot \ccf')_{\phi}\big)\\ = & \sum_{(m,\phi)\in \{0,\ldots,i\}\times S_{i-1,j}}(-1)^{m} \sgn(\phi) 
  \tau \big((\partial_{i}^{(m)}(\ccf) \cdot \ccf')_{\phi}\big).
  \end{align*}
   This last sum is of course $\partial_{i}(\ccf)\shuffle \ccf'$.

  \noindent ($s=4$) By an argument completely analogous to the case $s=3$, we get $$\sum_{(\ell,\sigma)\in A^{i,j}_{4}}(-1)^{\ell} \sgn(\sigma) 
  \,g_{\ell}(\ccf,\ccf',\sigma)=(-1)^{i+1} \ccf \shuffle\partial_{j}(\ccf').$$

  \smallskip
Summing all contributions in \ref{decomp}, we finally get $\partial_{n} ( \shuffle (\ccf \otimes \ccf' )) = \shuffle (\delta_n(\ccf \otimes \ccf')).$

\end{proof}

The following boundary formula spells out the exact signs implicit in 
\ref{prop:shuffle}.

\begin{Corollary}
  \label{lem:shuffle0}
  For $\ccf \in \widetilde{M}_{i}(\LLL)$ and $\ccf' \in \widetilde{M}_j(\LLL)$ 
  we have  
	$$\partial_{i+j+1}(\ccf \shuffle \ccf')=\partial_i (\ccf) \shuffle \ccf' + (-1)^{i+1} \, \ccf \shuffle \partial_j (\ccf' ).$$
\end{Corollary} 

Let $\pi:\redM_*(\PPP) \rightarrow \redC_*(\PPP)$ be the map from 
\ref{lem:norm}. By \ref{lem:norm} the map $\pi$ induces a chain
homotopy equivalence and in particular it is a chain map. 
This implies the following corollary, which further justifies that we
ignore actual multichains in our calculations, i.e. we implicitly always apply 
$\pi$ to shuffles.

\begin{Corollary} 
\label{cor:shuffleorderchain}
For $\ccf \in \widetilde{C}_{i}(\LLL)$ and $\ccf' \in \widetilde{C}_j(\LLL)$ 
  we have 
$$\partial_{i+j+1}\big(\,\pi(\ccf \shuffle \ccf')\,\big)=\pi\big(\,\partial_i (\ccf) \shuffle \ccf'\,\big) + (-1)^{i+1} \, \pi\big(\,\ccf \shuffle \partial_j (\ccf' )\,\big).$$
\end{Corollary} 
\section{The synor complex} \label{sec:synor}

  For this section we return to the situation where $\PPP$ is a poset, and 
  we consider an efficient way to understand the homology of all lower
  intervals through a subcomplex of its simplicial chain complex. 
  We do so using the following notions.
  
  An \defn{$i$-synor} is an element $x$ in $\PPP$ such that $\PPP_{<x}$ has
  nontrivial $(i-1)$st (reduced) homology, and a \defn{synor} is an element 
  that is an $i$-synor for some $i$. It is not hard to see that 
  $\PPP_{\mathrm{synors}}$, the subposet of synors, and $\PPP$ are 
  homologically equivalent, that is, the inclusion induces an isomorphism 
  of homology groups. This is an easy consequence of 
  a homological version of Quillen's Theorem A (see e.g., \cite[Corollary 4.3]{BWW}). Intuitively, this assures us that the poset of synors $\PPP_{\mathrm{synors}}$ already captures the homological information of $\PPP$. 
 
   We proceed by capturing the same homological information through a chain complex associated with $\PPP$, called a
  \defn{synor complex} of $\PPP$ and denoted by $(\mathcal{S}_*(\PPP),\delta_{*})$, which gives a more economical way of studying the homology of $\PPP$ than merely the subposet of synors. Eventually, we will think of $(\mathcal{S}_*(\PPP),\delta_{*})$
  as a subcomplex of
  the reduced simplicial chain complex $(\redC_*(\PPP),\partial_{*})$ of $\PPP$, whose order chains are all supported in $\PPP_{synors}$.
  In particular, the boundary operator of the synor complex will
  coincide with the simplicial boundary operator. 
  However, we first develop the ideas more abstractly:
  
  \medskip

  Let $\PPP$ be a poset.
  A $\PPP$-\defn{graded complex} $(C_*,\partial_*)$ is a complex of 
  vector spaces over a field $\KK$ 
  such that 
  \begin{itemize}
	\item $C_i = \bigoplus_{x \in \PPP} C_i^{(x)}$ for $i \geq 0$, 
        \item $\partial_i\big(\, C_i^ {(x)} \big) \subseteq \bigoplus_{y \leq x} C_{i-1}^{(y)}$ for $i \geq 1$ and 
        \item $C_{-1} = \KK$.  
  \end{itemize}

We call a $\PPP$-graded complex \defn{strictly} $\PPP$-\defn{graded} if 
$\partial_i\big(\, C_i^ {(x)}\,\big) \subseteq \bigoplus_{y < x} C_{i-1}^{(y)}$ for $i \geq 1$. 

The reduced simplicial chain complex $(\redC_*(\PPP),\partial_*)$ of a poset $\PPP$ is $\PPP$-graded with $\redC_i(\PPP)^{(x)}$, $i \geq 0$, 
being the $\KK$-vector space spanned by the order chains of cardinality $i+1$
and largest element $x$. We also have, $\redC(\PPP)_{-1} = \KK$. 
Note that except for trivial cases this chain complex is not strictly $\PPP$-graded.  

Let $(C_*,\partial)$ be a $\PPP$-graded chain complex and $\JJJ \subseteq \PPP$ 
an order ideal.
We set $C_i^\JJJ = \bigoplus_{x \in \JJJ} C_i^{(x)}$, $i \geq 0$ and
$C_{-1}^J= \KK$. Then the differential $\partial_*$ restricts to a differential 
on $C^\JJJ_*$ and therefore the following holds.

\begin{Lemma} 
  If $(C_*,\partial_*)$ is a $\PPP$-graded complex and
  $\JJJ \subseteq \PPP$ is an order ideal in $\PPP$, then 
  $(C_*^\JJJ,\partial_*|_{C_*^\JJJ})$ is a $\JJJ$-graded complex. 
\end{Lemma} 

If $(C_*,\partial_*)$ and $(D_*,\delta_*)$ are two $\PPP$-graded complexes, 
then a $\PPP$-\defn{graded chain map} $\phi : (C_*,\partial_*) \rightarrow (D_*,\delta_*)$ between 
$\PPP$-graded complexes is a map $\phi$ of chain complexes that satisfies 
$\phi(C_i^{(x)}) \subseteq D_i^ {(x)}$ for all $x \in \PPP, i \geq 0$ and also $\phi(C_{-1}) = D_{-1}$. 

A \defn{synor complex} $(\Sy(\PPP)_*,\delta_*)$ for the poset $\PPP$ is a 
strictly $\PPP$-graded complex together with an injective $\PPP$-graded chain map 
$\phi : (\Sy_*(\PPP),\delta_*) \rightarrow (\redC_*(\PPP),\partial_*)$ such that

\begin{itemize}
	\item[(S1)] $\dim_\KK \Sy_i(\PPP)^{(x)} = \dim_\KK\redH_{i-1}(\PPP_{< x})$ for every $x\in \PPP$.
  \item[(S2)] for every order ideal $\JJJ$ in $\PPP$ the restriction 
	  of $\phi$ to $\Sy_*(\PPP)^\JJJ$ induces an isomorphism between 
     $\Hom_*(\Sy(\PPP)^\JJJ)$ and $\redH_*(\JJJ)$.
\end{itemize}
We first prove the existence of such objects. We should mention that uniqueness is not guaranteed in general.
\begin{Proposition} \label{prop:synor} 
	For every poset $\PPP$ there exists a synor complex $(\Sy_*(\PPP),\delta_*)$
  for $\PPP$. 
\end{Proposition}
\begin{proof}
  We prove the claim by induction on $|\PPP|$. 

  If $|\PPP| \leq 1$ then we can take as
  $(\Sy_*(\PPP),\delta_*)$ the reduced simplicial chain complex of
$\PPP$ and set $\phi$ to be
  the identity.

  Now assume $|\PPP| \geq 2$. Let $x$ be a maximal element of $\PPP$. By induction there exists a synor complex
  $(\Sy_*(\PPP-x),\delta_*)$ for $\PPP-x$ together with the appropriate injective $(\PPP-x)$-graded chain map $\phi : 
  (\Sy_*(\PPP-x),\delta_*) \rightarrow (\redC_*(\PPP-x),\partial_*)$ inducing isomorphisms in homology.

  For each $i$, we set $\Sy_i(\PPP)^{(x)}$ to be the 
  $\KK$-vector space with basis indexed by symbols $\{x\star \zeta:\zeta \in \Zeta_{i-1}\}$, where $\Zeta_{i-1}$ denotes a set of cycle representatives of a basis of $\redH_{i-1}(\Sy(\PPP_{<x}))$.
  By use of (S2) for the ideal $\PPP_{<x}$ inside $\PPP-x$, we derive that $\Sy_*(\PPP)^ {(x)}$ satisfies 
  (S1). 

  We extend $\phi$ to each $\Sy_i(\PPP)^{(x)}$ by defining
  $\phi (x\star \zeta):=x*\phi(\zeta)$ for all $\zeta \in Z_{i-1}$ and extending linearly.
  
  
  Define $\Sy_*(\PPP)$ as the direct sum of $\Sy_*(\PPP-x)$ and $\Sy_*(\PPP)^{(x)}$
  and extend the differential to each $\Sy_i(\PPP)^{(x)}$ by
	setting $\delta_i(x\star\zeta) =  \zeta \in \Sy_{i-1}(\PPP-x)$ for all $\zeta\in Z_{i-1}$. It 
 follows immediately that $(\Sy_*(\PPP),\delta_*)$ is a strictly $\PPP$-graded complex and $\phi$ is 
  a $\PPP$-graded chain map 
  $(\Sy_*(\PPP),\delta_*) \rightarrow (\redC_*(\PPP),\partial_*)$, which is still injective.

  Now assume that we are given an order ideal $\JJJ$ in $\PPP$. If the order
  ideal does not contain $x$ then $\Sy_*(\PPP-x)^\JJJ = \Sy_*(\PPP)^\JJJ$ and (S2) 
  follows by induction.
  
  Let $\JJJ$ be an order ideal containing $x$. 
  Consider the quotient complex $(\Sy_*(\JJJ,\JJJ-x),\delta_{*})$ with $\Sy_i(\JJJ, \JJJ-x) = \Sy_i(\JJJ)/\Sy_i(\JJJ-x) \cong \Sy_i(\PPP)^{(x)}$ for all $i$. 
  We claim that the induced chain map $\phi':(\Sy_{*}(\JJJ,\JJJ-x),\delta_*)\rightarrow (\redC_{*}(\JJJ,\JJJ-x),\partial_*)$ defined by $$\phi'(x\star\zeta+\Sy_{i}(\JJJ-x))=\phi(x\star\zeta)+\redC_i(\JJJ-x)=x\ast\phi(\zeta)+\redC_i(\JJJ-x)$$ for every $i$, induces an isomorphism in homology. Since the boundary map of the complex $(\Sy_{*}(\JJJ,\JJJ-x),\delta_{*})$ is trivial, this is the same as saying that the induced map $\phi':\Sy_{i}(\JJJ,\JJJ-x)\cong \Sy_i(\PPP)^{(x)}\rightarrow \redH_i(\JJJ,\JJJ-x)$ is an isomorphism for every $i$.
  To see that this is true, we factor $\phi'$ as 
  \begin{figure}[H]
     
		\begin{tikzcd}
		      S_{i}(P)^{(x)} 
                \arrow{r}{\psi} \arrow{dr}{\phi'} 
            & 
                \redH_i(\JJJ_{\leq x},\JJJ_{<x}) \arrow{d}{j} \\ 
            &    
                \redH_i(\JJJ,\JJJ-x) 
            &
		\end{tikzcd}
  \end{figure}
 where $\psi:\Sy_i(\PPP)^{(x)}\rightarrow \redH_{i}(\JJJ_{\leq x},\JJJ_{<x})$ is the map sending $x\star \zeta$ to the representative $x\ast \phi(\zeta)$ for all $\zeta\in Z_{i-1}$, while $j$ is induced by inclusion of pairs. 
 The set $\{\phi(\zeta):\zeta\in Z_{i-1}\}$ considered as a set of cycles in $\redC_{i-1}(\JJJ_{<x})$ yields a basis of $\redH_{i-1}(\JJJ_{<x})$ by induction and (S2). Consequently, one can show that the relative cycles represented by $\{x*\phi(\zeta):\zeta\in Z_{i-1}\}$ yield a basis of $\redH_{i}(\JJJ_{\leq x},\JJJ_{<x})$.
	Since the map $\psi$ sends a basis to a basis, it is an isomorphism. 
        Finally, $j$ is also an isomorphism by excision or direct computation.

  \begin{figure}[H]
     \begin{centering}
		\begin{tikzcd}
			0 \arrow[r] & S_i(\JJJ-x) \arrow[d,"\phi"] \arrow[r] 
			& S_i(\JJJ) \arrow[d,"\phi"] \arrow[r] 
			& S_i(\JJJ,\JJJ-x) \arrow[r] \arrow[d,"\phi'"] & 0 \\
			0 \arrow[r] & \redC_i(\JJJ-x) \arrow[r] 
			& \redC_i(\JJJ) \arrow[r] 
			& \redC_i(\JJJ,\JJJ-x) \arrow[r] & 0 
		\end{tikzcd}
		\caption{Commutative diagram with exact rows}
		\label{fig:1} 
     \end{centering}
  \end{figure}
  \begin{figure}[H]
	\begin{centering}
		\begin{tikzcd}
			\arrow[r] 
			& \Hom_i(\Sy(\JJJ-x)) \arrow[d,"\phi"] \arrow[r] 
			& \Hom_i(\Sy(\JJJ)) \arrow[d,"\phi"] \arrow[r] 
			& \Hom_{i}(\Sy(\JJJ,\JJJ-x)) \arrow[d,"\phi'"] \arrow[r] & \,\\
			\arrow[r] 
			& \redH_i(\JJJ-x)  \arrow[r] 
			& \redH_i(\JJJ) \arrow[r] 
			& \redH_{i}(\JJJ,\JJJ-x) \arrow[r] & \,  
		\end{tikzcd}
		\caption{Induced commutative diagram of long exact sequences}
		\label{fig:2} 
	\end{centering}
	\end{figure}

  \ref{fig:1} induces the diagram \ref{fig:2} of long exact sequences.
  Since the left $\phi$ map and $\phi'$ are isomorphisms, it follows by the five lemma that the
  middle $\phi$ map is an isomorphism as well.
  Now (S2) follows.
\end{proof}
\medskip

The next corollary is now immediate.

\begin{Corollary}
	\label{relaxrel}
	Let $\PPP$ be a poset and $\JJJ \subseteq \PPP$ an order ideal. 
	For a synor complex $(\Sy_*(\PPP),\delta_*)$ for $\PPP$ the inclusion
	of $\Sy_*(\PPP)$ into $\redC_*(\PPP)$ induces an isomorphism in
	homology of $\Hom_*(\Sy(\PPP),\Sy(\JJJ))$ and $\redH_*(\PPP,\JJJ)$. 
\end{Corollary}
\begin{proof}
	Consider the following commutative diagram of chain groups with exact rows.

     \begin{center}
		\begin{tikzcd}
			0 \arrow[r] & S_i(\JJJ) \arrow[d,"\phi"] \arrow[r] 
			& S_i(\PPP) \arrow[d,"\phi"] \arrow[r] 
			& S_i(\PPP,\JJJ) \arrow[r] \arrow[d,"\phi'"] & 0 \\
			0 \arrow[r] & \redC_i(\JJJ) \arrow[r] 
			& \redC_i(\PPP) \arrow[r] 
			& \redC_i(\PPP,\JJJ) \arrow[r] & 0 
		\end{tikzcd}
     \end{center}

	The two leftmost vertical arrows induce isomorphisms in homology 
	by \ref{prop:synor}. By the five-lemma  this then implies that the rightmost map
	induces an isomorphism of $\Hom_*(\Sy(\PPP),\Sy(\JJJ))$ and $\Hom_*(C(\PPP),C(\JJJ))= \Hom_*(\PPP,\JJJ)$. 
\end{proof}

The following property of a synor complex is obviously true in the case of 
the synor complex constructed in the
proof of \ref{prop:synor}. But the lemma guarantees that it must in fact hold
for any synor complex. To state it, we set up some notation that will appear frequently from now on.

For $x \in \PPP$ we call an element of $\Sy_*(\PPP)^{(x)}$ a 
\defn{principal synor chain} of $\PPP$.
For a principal synor chain $\gamma \in \Sy_{i}(\PPP)^{(x)}$ it follows
that $\phi(\gamma)$ is a linear combination of order chains
$c = (x=x_0 > \cdots > x_i)$. 
For $0 \leq \ell \leq i$ we write 
$X_\ell^\gamma$ for the collection of all order chains 
$(x_0 > \cdots > x_{\ell})$ for which there is a chain
$c = (x_0 > \cdots > x_{\ell} > \cdots > x_i)$ in the support
of $\phi(\gamma)$. For any chain $c = (x_0 > \cdots > x_i)$ we write $\min(c)$ for $x_i$.

\begin{Lemma}
  \label{lem:lrepresentation}
  Let $x \in \PPP$ and let $\gamma \in \Sy_i(\PPP)^{(x)}$ be a principal synor chain.
  Then for $0 \leq \ell \leq i$ we can write 
  $\phi(\gamma) = \sum_{\chi \in X_\ell^\gamma} \chi*\phi(\zeta_\chi)$ for 
 cycles $\zeta_\chi \in \bigoplus_{y < min(\chi)} \Sy_{i-\ell-1}(\PPP)^ {(y)}$.
\end{Lemma}
\begin{proof}
  We proceed by induction on $\ell$. If $\ell = 0$ then 
  by $\phi(\gamma) \in \redC_i(\PPP)^{(x)}$ it follows that
  $\phi(\gamma) = x* \zeta$ and $X_0^\gamma = \{\,(x)\,\}$. We then have
  $$\phi(\delta_i(\gamma)) = \partial_i(\phi(\gamma)) = \partial_i( x*\zeta) = \zeta - x*\partial_{i-1}(\zeta).$$
  Since the synor complex is strictly $\PPP$-graded and $\phi$ is a $\PPP$-graded chain map, it follows
  that $\partial_{i-1}(\zeta) =  0$ and thus $\phi(\gamma)=x*\zeta=x*\phi(\delta_{i}(\gamma))$ is precisely of the desired form, because the cycle $\delta_{i}(\gamma)$ is clearly in $ \bigoplus_{y < x} \Sy_{i-1}(\PPP)^ {(y)}$.

  Now assume that for some $\ell < i$ we have
  $\phi(\gamma) = \sum_{\chi \in X_\ell^\gamma} \chi*\phi(\zeta_\chi)$,
  where each
  $\zeta_\chi$ is a cycle in $\bigoplus_{y < \min(\chi)} \Sy_{i-\ell-1}(\PPP)^{(y)}$. It follows that
  $\phi(\zeta_\chi) = \sum_{y < min(\chi)} \phi(\gamma_{\chi,y})$, where each $\gamma_{\chi,y}$ is a principal 
  $(i-\ell-1)$-synor chain in $\Sy_{i-\ell-1}(\PPP)^{(y)}$. By the induction basis we have that $\phi(\gamma_{\chi,y}) = y * \zeta_{\chi,y}$ for some $(i-\ell-2)$-cycle $\zeta_{\chi,y}\in \bigoplus_{z < y} \Sy_{i-\ell-2}(\PPP)^{(z)}$. Injectivity of $\phi$ easily implies that
  $X_{\ell+1}^\gamma = \big\{\, \chi*y~\big|~
  \chi \in X_\ell^\gamma, \zeta_{\chi,y} \neq 0\,\big\}$.
  But then 
  $$\phi(\gamma) = \sum_{\chi \in X_\ell^\gamma}\chi*\big(  \sum_{y < \min(\chi)} y*\phi(\zeta_{\chi,y})\,\big).$$
  Setting $\zeta_{\chi'} = \zeta_{\chi,y}$ for $\chi' = \chi * y$
  then yields
  $$\phi(\gamma)=\sum_{\chi' \in X_{\ell+1}^\gamma} \chi' * \phi(\zeta_{\chi'})$$ where the $\zeta_{\chi'}$ are clearly of the desired form.
  This completes the induction step.
\end{proof}

We call a representation $\phi(\gamma) = \sum_{\chi \in X_\ell^\gamma} \chi*\phi(\zeta_\chi)$ of a principal synor chain $\gamma$ an \defn{$\ell$-representation} of $\gamma$. 
Shifting our attention away from the embedding $\phi$, we may identify the synor complex $(\Sy_*(\PPP),\delta_*)$ with its image under $\phi$. Thus,
from now on we will always consider $(\Sy_*(\PPP),\delta_*)$ as a 
subcomplex of $(\redC_*(\PPP),\partial_*)$. Then, an
$\ell$-representation of the principal synor chain $\gamma$ can be written as
$\gamma = \sum_{\chi \in X_\ell^\gamma} \chi*\zeta_\chi$ where each $\zeta_{\chi}$ is a cycle in $\bigoplus_{y < min(\chi)} \Sy_{i-\ell-1}(\PPP)^ {(y)}$.

\section{Shuffling elements into cycles}
\label{sec:shuffling}
In this section we introduce a sequence of lemmas for future reference. The first two (\ref{shufflethatshit},\ref{chi}) concern a vanishing property associated with the synor complex. The third (\ref{ro}) provides a useful map that sends general order chains to synor chains. The fourth, \ref{lem:homologous}, concerns an obvious result in poset homology.

To initiate the discussion, let $\ccf = (x_0 > \cdots > x_m), \ccf' = (x_0' > \cdots > x_m')$ be 
two order chains in $\PPP$.
For $0 \leq j \leq m$ we set $\ccf \sim_j \ccf'$ if
$x_i = x_i'$ for $i \neq j$ and $0 \leq i \leq m$. 
For a fixed order chain $\ccf = (x_0 > \cdots > x_m)$ in $\PPP$
and an $m$-chain $\tau = \sum_{\ell=1}^k \lambda_{\ell} \ccf_{\ell} \in \redC_m(\PPP)$, 
we write $[\tau:\ccf]_j$ for $\displaystyle{\sum_{\ccf_{\ell} \sim_j \ccf} \lambda_{\ell}}$. 

The following property of synor cycles distinguishes them from general simplicial cycles:

\begin{Lemma}
  \label {shufflethatshit}
  If $\zeta \in \Sy_m(\PPP)$ is a cycle and $\ccf\in\redC_m(\PPP)$ then $[\zeta:\ccf]_j = 0$ for $j=0, \ldots, m$. 
\end{Lemma}
\begin{proof}
  We proceed by induction on $m$. 
  
  If $m = 0$ then $j=0$ and $\bar{\ccf} \sim_j \ccf$ for all $0$-order chains $\bar{\ccf}$. Thus, writing $\zeta=\sum_{\ell=0}^{k}\lambda_{\ell}\,(v_{\ell})$  yields  $0=\partial_{0}(\zeta)=\sum_{\ell=0}^{k}\lambda_{\ell}=[\zeta:\ccf]_{0}$ as desired.

  Assume now that $m>0$.
  By the previous section, we can write $$\zeta=\sum_{x\in \PPP} x*\zeta_{x}$$ for $(m-1)$-cycles $\zeta_x \in \bigoplus_{y<x}\Sy_{m-1}(\PPP)^{(y)}$. Note that for $x\in \PPP\setminus X^{\zeta}_{0}$, the corresponding cycle $\zeta_{x}$ is of course zero.

  Set $\ccf' = (x_{1} > \cdots > x_{m})$, where $\ccf=(x_{0} > \cdots > x_{m})$. 
  Consider the case $1 \leq j \leq m$, for which we have $[\zeta:\ccf]_j = [\zeta_{x_{0}}:\ccf']_{j-1}$.
  Applying the inductive hypothesis to $\zeta_{x_{0}}\in\Sy_{m-1}(\PPP)$ and $c'\in \redC_{m-1}(\PPP) $, we get $[\zeta_{x_{0}}:\ccf']_{j-1} = 0$ which implies $[\zeta:\ccf]_j =0$.

  It remains to consider the case $j = 0$.  
  Since $\partial_m(x*\zeta_x) = \zeta_x$ and since $\zeta$ is a cycle,
  it follows that $\sum_{x \in \PPP} \zeta_x = 0$.
  In particular, if $\lambda_x$ is the coefficient of the term 
  $\ccf'$ appearing in $\zeta_x$, then $\sum_{x \in \PPP} \lambda_x = 0$. 
	Note that we do not exclude the possibility $\lambda_{x}=0$. Observe then that $[x * \zeta_x:\ccf]_0 = \lambda_x$, which gives 
   $[\zeta:\ccf]_0 = \sum_{x\in \PPP}\lambda_{x} = 0$ as desired.
\end{proof}

We now use \ref{shufflethatshit} to prove a surprising vanishing result associated with principal synor chains.
To state it, we let $\gamma\in\Sy_m(\PPP)^{(x)}$ be a principal synor chain with $0$-representation $\gamma=x*\zeta$.
Consider the equivalence classes $X_\ell^\gamma/\sim_j$ of the relation $\sim_j$. 
We write $[\chi]$ for the class of $\chi \in X_\ell^\gamma$.









\begin{Lemma}\label{chi}
   Let\footnote[9]{Note that the result is clearly false for $j=0$.} $1 \leq j \leq  \ell \leq m$ and let $\gamma = x*\zeta \in \Sy_m(\PPP)^{(x)}$ with $\ell$-representation
   \[\gamma =  \sum_{\chi \in X_\ell^\gamma} \chi * \zeta_{\chi}.\]
   Then, for any $[\chi] \in X_\ell^ \gamma/\sim_j$ we have 
	$$\sum_{\chi' \in [\chi]} \zeta_{\chi'} = 0.$$
\end{Lemma}
\begin{proof}
   We can write $\gamma$ as 
   \begin{align*} 
     \gamma& =\sum_{\overline{\chi} \in X_m^\gamma} \lambda_{\overline{\chi}} \overline{\chi} 
   \end{align*}
   where $\lambda_{\overline{\chi}}$ are scalars.
   We decompose the order chains $\overline{\chi}=(\overline{x_{0}} > \cdots > \overline{x_{m}})\in X^{\gamma}_{m}$ into $\overline{\chi} = \chi*\chi_0$, where $\chi=(\overline{x_{0}} > \cdots > \overline{x_{l}})$ and $\chi_{0}=(\overline{x_{l+1}} > \cdots > \overline{x_{m}})$. We take two such chain decompositions $\overline{\chi} = \chi*\chi_0$, $\overline{\chi}' = \chi'*\chi_0'$  and note that $j\leq \ell$ implies
   $\overline{\chi} \sim_j \overline{\chi}'$ if and only if 
   $\chi \sim_j \chi'$ and $\chi_0 = \chi_0'$.
   
   Now $\gamma = x * \zeta$ and $\zeta$ is a cycle.
   Thus any $\overline{\chi} \in X_m^\gamma$ can be written 
   as $x * \partial_{m}^{(0)}(\overline{\chi})$ where 
   $\partial_{m}^{(0)}(\overline{\chi})= \overline{\chi} \setminus \{x\} \in X_{m-1}^\zeta$.
   
   It follows by \ref{shufflethatshit} that $0 = [\zeta:\partial_{m}^{(0)}(\overline{\chi})]_{j-1} = [\gamma:\overline{\chi}]_{j}$. 
   These facts imply  that for any $\overline{\chi} = \chi * \chi_0 \in X_m^\gamma$ with $\chi \in X_\ell^ \gamma$ we have
   that 
   \begin{align}
   \label{eq:0j} 
   0 & = [\gamma:\overline{\chi}]_j = 
   \sum_{\chi' \in [\chi]\in X_\ell^\gamma/\sim_j} \lambda_{\chi' * \chi_0}.
   \end{align}
   On the other hand $$\zeta_{\chi'} = \sum_{\genfrac{}{}{0pt}{}{\chi_0'}{\chi'*\chi_0' \in X_m^ \gamma}} \lambda_{\chi'*\chi_0'} \chi_0'.$$ 
   It follows that
   \begin{align*} 
     \sum_{\chi' \in [\chi]} \zeta_{\chi'} & = \sum_{\chi' \in [\chi]}\sum_{\genfrac{}{}{0pt}{}{\chi_0'}{\chi'*\chi_0' \in X_m^ \gamma}} \lambda_{\chi'*\chi_0'} \chi_0' \\
     & = \sum_{\chi_0'} \big( \,\sum_{\genfrac{}{}{0pt}{}{\chi' \in [\chi]}{\chi'*\chi_0' \in X_m^ \gamma}} \lambda_{\chi'*\chi_0'}\,\big)\, \chi_0' \\
     & \overset{\text{\eqref{eq:0j}}}{=} 0 
     \end{align*}
\end{proof}
Next, we construct a map $\rho:\redC(\PPP)\rightarrow \Sy(\PPP)$ whose various properties we use freely in the next section. 
\begin{Lemma}\label{ro}
  Let $\PPP$ be a poset.
  For each $k \geq -2$ there is a linear map 
	\[\rho_k:\redC_k(\PPP) \ \longrightarrow\ 
	\mathcal{S}_{k}(\PPP)\]
  such that
  \begin{itemize}
    \item[(R0)] $\rho_{-1}$ sends $\emptyset$ to itself, viewed as a 
		  $(-1)$-synor cycle.
    \item[(R1)] For all $k \geq -1$, we have $\rho_{k-1} \circ \partial_k = 
		\partial_k \circ \rho_{k}$
    \item[(R2)] For all $k \geq 0$ and for all order chains
	$c = (x_0 > \cdots > x_k)$, we have that $\rho_k(c)$ is 
	supported in $\PPP_{\leq x_0}$.
  \end{itemize}
\end{Lemma}
\begin{proof}
	We define $\rho_k$ on the basis of $\redC_k(\PPP)$ given by 
  order chains $c = (x_0 > \cdots > x_k)$ 
  and then extend linearly.
  
  By (R0) the map $\rho_{-1}$ is defined. 
	Since $C_{-2}(\PPP)=\Sy_{-2}(\PPP)=0$, 
  the map $\rho_{-2}$ is trivial. 
  By $\partial_{-1} (\emptyset) = 0$ it follows that 
  $\rho_{-2} \circ \partial_{-1} = \partial_{-1} \circ \rho_{-1}$
  and hence (R1) is satisfied for $k=-1$. 
  (R2) is trivially satisfied for $k =-1$. 

  Assume $\rho_k$ has been defined for some $k \geq -1$ and (R1) and (R2) 
  hold for $k$. 
  To define $\rho_{k+1}(c)$ for any order chain 
  $c = (x_{0} > \cdots > x_{k+1})$ 
  we compute $\partial_k \circ \rho_k \circ \partial_{k+1}(c) 
  \overset{\text{(R1)}}{=} \rho_{k-1} \circ \partial_k \circ \partial_{k+1} (c) = 0$.
  Using (R2) it follows that $\rho_k\circ\partial_{k+1}(c)$ is a synor cycle supported on $\PPP_{\leq x_{0}}$. 
  By contractibility of $\PPP_{\leq x_{0}}$ we can use 
  \ref{relaxrel} to set $\rho_{k+1}(c)$ equal to a synor chain in 
  $\PPP_{\leq x_{0}}$ whose boundary is 
  $\rho_k\circ\partial_{k+1}(c)$.
  Now (R1) holds by $\partial_{k+1} \circ \rho_{k+1} (c) = 
  \rho_k\circ \partial_{k+1}(c)$ and linearity and (R2) holds by construction.
\end{proof}

Finally, the next elementary lemma will allow us to prove certain homologous equivalences by taking boundaries. Formally:
\begin{Lemma}
 \label{lem:homologous}
  Let $\PPP$ be a poset with unique maximal element $\hat{1}$.
	Let $\gamma, \gamma' \in \redC_{m}(\PPP)$ where
 $\gamma = \hat{1}*\zeta$ for a cycle $\zeta
	\in \redC_{m-1}(\PPP_{< \hat{1}})$. Then
 the following are equivalent:
 \begin{itemize}
	 \item[(i)] $\gamma$ and $\gamma'$ are homologous in $\redH_{m}(\PPP,\PPP_{< \hat{1}})$.
   \item[(ii)] $\partial_m(\gamma)=\zeta$ and $\partial_m(\gamma')$ are
	   homologous in $\redH_{m-1}(\PPP_{< \hat{1}})$. 
 \end{itemize}
\end{Lemma}

\begin{proof} 
    \noindent (i) $\Rightarrow$ (ii)

      First notice that $\gamma'\in \redH_{m}(\PPP,\PPP_{<\hat{1}})$ implies that $\partial_{m}(\gamma')$ is supported in $\PPP_{<\hat{1}}$ (whereas the analogous result for $\gamma$ is clear).
      
      Let $c \in \redC_{m+1}(\PPP) \ \text{and} \ \sigma \in \redC_{m}(\PPP_{< \hat{1}}) $ be such that
      $\gamma-\gamma'=\partial_{m+1}(c)+\sigma$.
      Taking boundaries yields $\partial_{m}(\gamma)-\partial_{m}(\gamma')=\partial_{m}(\sigma)$. Since $\partial_{m}(\sigma)$ is a boundary in $\PPP_{<\hat{1}}$, this implies the stated homologous equivalence  in $\redH_{m-1}(\PPP_{<\hat{1}})$.

\noindent (ii) $\Rightarrow$ (i)

	   Let $\sigma \in \redC_{m}(\PPP_{< \hat{1}})$ 
      be such that $ \partial_m(\gamma) - \partial_m(\gamma')=\partial_{m}(\sigma)$.
      Thus, the chain $\gamma-\gamma'-\sigma\in \redC_{m}(\PPP)$ is a cycle. The contractibility of $\PPP$ implies it is also a boundary: There exists some $c\in \redC_{m+1}(\PPP)$ with $\gamma-\gamma'-\sigma=\partial_{m+1}(c)$. Rewriting this as $\gamma-\gamma'=\partial_{m+1}(c)+\sigma$, we get the homologous equivalence in $\redH_{m}(\PPP,\PPP_{<\hat{1}})$.
\end{proof}

\section{Synor representations and the proof of the main theorem}

In this section, we return to the case of a lattice $\LLL$ with maximal element $\hat{1}$ and minimal element $\hat{0}$. We will often work with the associated posets  $\overline{\LLL}=\LLL\setminus \{\hat{0}\}$ and $\check{\LLL}=\LLL\setminus\{\hat{1},\hat{0}\}$. 

In view of \ref{mainthm}, we will assume that the top element $\hat{1}$ is an $m$-synor of $\overline{\LLL}$ for some appropriate $m$. By $\dim_\KK \Sy_m(\overline{\LLL})^{(\hat{1})} = \dim_\KK\redH_{m-1}(\check{\LLL})\neq 0$, this allows us to pick a principal $m$-synor chain $\gamma=1*\zeta\in \Sy_m(\overline{\LLL})^{(\hat{1})}$. We will modify the $\ell$-representations of $\gamma$ in order to extract some innocent-looking but powerful homological information in \ref{lmuff}. In doing so, we will need most of the material developed in previous sections.

We begin with a trivial observation about consecutive representations of $\gamma$.
\begin{Lemma}
\label{lem:ellellm1}
  Let $1 \leq \ell \leq m$ and let $\gamma=\hat{1}*\zeta\in\Sy_m(\overline{\LLL})^{(\hat{1})}$ be a principal $m$-synor chain with $\ell$- and $(\ell-1)$-representations 
  $$\gamma = \sum_{\chi \in X_\ell^ \gamma} \chi * \zeta_\chi =
  \sum_{\chi' \in X_{\ell-1}^ \gamma} \chi' * \zeta_ {\chi'}.$$
  Then
  $$\zeta_{\chi'} = \sum_{ \genfrac{}{}{0pt}{}{\chi \in X_\ell^{\gamma}}{\partial_\ell^ {(\ell)} (\chi) = \chi'}} \min(\chi) * \zeta_\chi.$$
\end{Lemma}  
The following less trivial result prepares us for \ref{lmuff}:
\begin{Lemma} \label{lem:step}
 Let $\gamma=\hat{1}*\zeta\in\Sy_m(\overline{\LLL})^{(\hat{1})}$ be a principal $m$-synor chain  and for some $1 \leq \ell \leq m$ let  $\gamma = \sum_{\chi \in X_\ell^ \gamma} \chi * \zeta_\chi$ be  its $\ell$-representation. 
  Then $$\sum_{\chi \in X_\ell^ \gamma} \rho_{\ell-1}(\partial_\ell^ {(0)}(\chi)) \shuffle \zeta_\chi$$ and
  $$\sum_{\chi' \in X_{\ell-1}^ \gamma} \rho_{\ell-2}(\partial_{\ell-1}^ {(0)}(\chi')) \shuffle \zeta_{\chi'}$$
  are homologous in $\redH_{m-1}(\check{\LLL})$.
  \end{Lemma}
\begin{proof}
We first need to check that both expressions yield cycles in $\redC_{m-1}(\check{\LLL})$. Since the argument is the same in both cases, we confine ourselves to the proof for $\sum_{\chi \in X_\ell^ \gamma} \rho_{\ell-1}(\partial_\ell^ {(0)}(\chi)) \shuffle \zeta_\chi$:

To see that this is supported in $\check{\LLL}$, we may use (R2) of \ref{ro} together with the fact that the shuffle product is supported in joins.

To see that it is a cycle, we first use the boundary formula for the shuffle product, \ref{cor:shuffleorderchain}, together with (R1) from \ref{ro}: \begin{align*}
\partial_{m-1}\big(\,\sum_{\chi \in X_\ell^ \gamma} \rho_{\ell-1}(\partial_\ell^ {(0)}(\chi))  & \shuffle \zeta_\chi\,\big)=\sum_{\chi \in X_\ell^ \gamma} \partial_{\ell-1}\big(\,\rho_{\ell-1}(\partial_\ell^ {(0)}(\chi))\,\big) \shuffle \zeta_\chi\\&\overset{\text{(R1)}}{=}\sum_{\chi \in X_\ell^ \gamma} \rho_{\ell-2}\big(\,\partial_{\ell-1}(\partial_\ell^ {(0)}(\chi))\,\big) \shuffle \zeta_\chi\\&=\sum_{j=0}^{l-1}(-1)^{j}\sum_{\chi \in X_\ell^ \gamma} \rho_{\ell-2}\big(\,\partial_{\ell-1}^{(j)}(\partial_\ell^ {(0)}(\chi))\,\big) \shuffle \zeta_\chi\end{align*}
We claim that for each $j$, the corresponding term $\displaystyle{\sum_{\chi \in X_\ell^ \gamma}} \rho_{\ell-2}\big(\,\partial_{\ell-1}^{(j)}(\partial_\ell^ {(0)}(\chi))\,\big) \shuffle \zeta_\chi$ vanishes.
Indeed, 
\begin{align*} 
  \sum_{\chi \in X_\ell^ \gamma} \rho_{\ell-2} & \big(\,\partial_{\ell-1}^{(j)}(\partial_\ell^ {(0)}(\chi))\,\big) \shuffle \zeta_\chi \\
  & =\sum_{[\chi] \in X_\ell^\gamma/\sim_{j+1}} 
  \Big( \,\rho_{\ell-2}\big(\,\partial_{\ell-1}^{(j)}(\partial_\ell^{(0)}(\chi)) \,\big)
  \shuffle  \sum_{\chi' \in [\chi]} \zeta_{\chi'}\Big) =0,\end{align*} where in the last step we used \ref{chi}. Importantly, $1\leq j+1$ makes the application of the Lemma legitimate.\medskip

  The proof of the homologous equivalence is similar:
  
  Clearly,
  \begin{align}
  \nonumber \sum_{\chi \in X_\ell^ \gamma} \rho_{\ell-1}(\partial_\ell^ {(0)}(\chi)) \shuffle \zeta_\chi & = \, \sum_{\chi \in X_\ell^ \gamma} \rho_{\ell-1}(\partial_\ell^ {(0)}(\chi)) \shuffle \partial_{m-\ell}\big(\,\min(\chi) * \zeta_\chi \,\big)\\
  \label{eq:twoterm} = & \,(-1)^{\ell}\partial_m\big(\,\sum_{\chi \in X_\ell^\gamma} 
  \rho_{\ell-1}\big(\,\partial_\ell^{(0)} (\chi) \,\big) \shuffle (\min(\chi) * \zeta_\chi)\,\big) + \\
  &\label{eq:twoterms} \,(-1)^{\ell+1}\sum_{\chi \in X_\ell^\gamma} 
  \partial_{\ell-1}\big(\, \rho_{\ell-1}(\partial_\ell^{(0)}(\chi))\, \big) \shuffle (\min(\chi) * \zeta_\chi)
  \end{align}
  Arguing as before, we get that the first term on the right hand side of
  \eqref{eq:twoterm} is a boundary in $\check{\LLL}$. 
  Now consider the second term \eqref{eq:twoterms}: We have
  \begin{align*} 
  & (-1)^{\ell+1}\sum_{\chi \in X_\ell^\gamma} 
  \partial_{\ell-1}\big(\,\rho_{\ell-1}(\partial_\ell^{(0)}(\chi))\,\big) \shuffle (\min(\chi)* \zeta_\chi) \\
  & \overset{\text{(R1)}}{=}(-1)^{\ell+1}\sum_{\chi \in X_\ell^\gamma} 
  \rho_{\ell-2}\big(\,\partial_{\ell-1}(\partial_\ell^{(0)}(\chi))\,\big) \shuffle (\min(\chi) * \zeta_\chi) \\
  & =\sum_{j=0}^{\ell-1} (-1)^{\ell+1+j} \sum_{\chi \in X_\ell^\gamma} 
  \rho_{\ell-2}\big(\,\partial_{\ell-1}^{(j)}(\partial_\ell^{(0)}(\chi))\, \big) \shuffle (\min(\chi)*\zeta_\chi).
  \end{align*}
  
  We treat the summands for $j < \ell-1$ and $j = \ell-1$ 
  separately. 
  
  For $j < \ell-1$, each term vanishes as before, since we can again make the terms $\sum_{\chi'\in[\chi]}\zeta_{\chi'}$ appear and apply \ref{chi}.

For $j = \ell-1$, we obtain by \ref{lem:ellellm1} that
\begin{align*} (-1)^{\ell+1+\ell-1}\sum_{\chi \in X_\ell^\gamma} 
  \rho_{\ell-2}\big(\,\partial_{\ell-1}^{(\ell-1)}(\partial_\ell^{(0)}(\chi))\,\big) & \shuffle (\min(\chi) * \zeta_\chi) \\ = \sum_{\chi' \in X_{\ell-1}^\gamma} \rho_{\ell-2}(\partial_{\ell-1}^{(0)}(\chi')) \shuffle \zeta_{\chi'}.
  \end{align*}
  Thus, we have managed to express the difference of cycles $$\sum_{\chi \in X_{\ell}^\gamma} \rho_{\ell-1}(\partial_{\ell}^{(0)}(\chi)) \shuffle \zeta_{\chi} -\sum_{\chi' \in X_{\ell-1}^\gamma} \rho_{\ell-2}(\partial_{\ell-1}^{(0)}(\chi')) \shuffle \zeta_{\chi'}$$ as a term \ref{eq:twoterm}, which is a boundary in $\check{\LLL}$.

\end{proof}
Now to the relative version:
\begin{Proposition}\label{lmuff}
  Let $\gamma = \hat{1} * \zeta\in\Sy_m(\overline{\LLL})^{(\hat{1})}$ be a principal $m$-synor chain 
  and for some $1 \leq \ell \leq m$ let 
  $\gamma\, =\ \sum_{\chi \in X_\ell^\gamma} \chi*\zeta_\chi$ 
  be its $\ell$-representation.

  Then $\gamma$ is homologous to 
  $$\sum_{\chi \in X_\ell^\gamma} \rho_\ell(\chi) \shuffle \zeta_\chi$$ 
  in $\redH_{m}(\overline{\LLL},\check{\LLL})$.
\end{Proposition}
\begin{proof}
  Let $\gamma' = \sum_{\chi \in X_\ell^\gamma} \rho_\ell(\chi) \shuffle \zeta_\chi$.
  By \ref{lem:homologous} it suffices to show that $\partial_m(\gamma)$ and
  $\partial_m(\gamma')$ are homologous as cycles in $\redH_{m-1}(\check{\LLL})$. We first check that both boundary terms are supported in $\check{\LLL}$.
  
  For $\partial_m(\gamma) = \zeta$ this is clear.  
  
  For $\partial_{m}(\gamma')$ we follow a similar route to the proof of \ref{lem:step}: We compute
  \begin{eqnarray*}
\partial_{m}\big(\,\sum_{\chi \in X_\ell^ \gamma} \rho_{\ell}(\chi) \shuffle \zeta_\chi\,\big)&=&\sum_{\chi \in X_\ell^ \gamma} \partial_{\ell}(\rho_{\ell}(\chi)) \shuffle \zeta_\chi\\&\overset{\text{(R1)}}{=}&\sum_{\chi \in X_\ell^ \gamma} \rho_{\ell-1}(\partial_{\ell}(\chi)) \shuffle \zeta_\chi\\&=&\sum_{j=0}^{\ell}(-1)^{j}\sum_{\chi \in X_\ell^ \gamma} \rho_{\ell-1}(\partial_{\ell}^{(j)}(\chi)) \shuffle \zeta_\chi.\end{eqnarray*}

For each $j>0$, the corresponding term $$\sum_{\chi \in X_\ell^ \gamma} \rho_{\ell-1}(\partial_{\ell}^{(j)}(\chi)) \shuffle \zeta_\chi=\sum_{[\chi] \in X_\ell^\gamma/\sim_{j}} 
  \Big( \,\rho_{\ell-1}\big(\,(\partial_\ell^{(j)}(\chi)) \,\big)
  \shuffle  \sum_{\chi' \in [\chi]} \zeta_{\chi'}\Big) $$ vanishes by \ref{chi}.

  Then the remaining term $\sum_{\chi \in X_\ell^ \gamma} \rho_{\ell-1}(\partial_{\ell}^{(0)}(\chi)) \shuffle \zeta_\chi$ is supported in $\check{\LLL}$ by (R2) of \ref{ro}.
  \medskip
  
  Now let us prove the homologous equivalence.
  The above argument shows that $$\partial_{m}(\gamma')=\sum_{\chi \in X_\ell^ \gamma} \rho_{\ell-1}(\partial_{\ell}^{(0)}(\chi)) \shuffle \zeta_\chi.$$
  Applying \ref{lem:step} iteratively, namely for $\ell'=\ell,\ell-1,\ldots,1$, we get that the terms $$\sum_{\chi \in X_{\ell-s}^ \gamma} \rho_{\ell-s-1}(\partial_{\ell-s}^{(0)}(\chi)) \shuffle \zeta_{\chi}$$ belong to the same homology class in $\redH_{m-1}(\check{\LLL})$, for $s=0,1,\ldots,\ell$. 
  Observe that the term corresponding to $s=0$ is precisely $\partial_{m}(\gamma')$, whereas the term corresponding to $s=\ell$ is $\partial_{m}(\gamma)$:
  This is because, for $s=\ell$ we have $X_{\ell-s}^ \gamma=X_{0}^ \gamma=\{\hat{1}\}$, $\zeta_{\hat{1}}=\zeta$ and so by (R0) of \ref{ro}: $$\sum_{\chi \in X_{\ell-s}^ \gamma} \rho_{\ell-s-1}(\partial_{\ell-s}^{(0)}(\chi)) \shuffle \zeta_{\chi}= \rho_{-1}(\partial_{0}^{(0)}(\hat{1})) \shuffle \zeta_{\hat{1}}=\rho_{-1}(\emptyset)\shuffle \zeta = \zeta = \partial_{m}(\gamma).$$ This of course proves the stated homologous equivalence.


%



\end{proof}

We are now in position to formulate and prove our main result on
lattice homology.

\begin{Theorem} \label{mainthm}
  Let $\LLL$ be a lattice and let $1\leq i_1,i_2$ and $0\leq k\leq\min\{i_1,i_2\}$ be integers.
  If $\hat{1} \in \overline{\LLL}$ is an 
  $(i_1 + i_2 - k - 1)$-synor, then there is an $(i_1-1)$-synor $x$ of $\overline{\LLL}$  and
  an $(i_2-1)$-synor $y$ of $\overline{\LLL}$ such that $\hat{1} = x \vee y$.
\end{Theorem}

\begin{proof}
  Since $\hat{1}$ is an $(i_1 + i_2  - k -1)$-synor, we have $\redH_{i_{1}+i_{2}-k-1}(\overline{\LLL},\check{\LLL})\cong \redH_{i_{1}+i_{2}-k-2}(\check{\LLL}) \neq 0.$ By \ref{relaxrel} we deduce $H_{i_{1}+i_{2}-k-1}(\Sy(\overline{\LLL}),\Sy(\check{\LLL}))\neq 0$, so we may choose a principal synor chain $\gamma=\hat{1}*\zeta\in\Sy_{i_{1}+i_{2}-k-1}(\overline{\LLL})^{(\hat{1})}$ which represents a nontrivial homology class in $H_{i_{1}+i_{2}-k-1}(\Sy(\overline{\LLL}),\Sy(\check{\LLL}))$. \ref{relaxrel} (for the inclusion map) implies that $\gamma$, viewed as a cycle in $\redH_{i_{1}+i_{2}-k-1}(\overline{\LLL},\check{\LLL})$ is also nontrivial.
  
  Let $$\gamma=\sum_{\chi \in X_{i_1-1}^\gamma} \chi * \zeta_\chi$$
    be the $(i_1-1)$-representation of $\gamma$. Note that to carry out this representation legitimately, we need $k\leq i_{2}$, for otherwise the terms $\zeta_{\chi}$ do not make sense.
    By \ref{lmuff},  $\gamma$ is homologous to 
    \begin{equation}\label{firstnontr}\sum_{\chi \in X_{i_1-1}^\gamma} \rho_{i_1-1}(\chi) \shuffle \zeta_\chi\end{equation} in $\redH_{i_{1}+i_{2}-k-1}(\overline{\LLL},\check{\LLL})$.
    
    Since $\gamma$ is nontrivial, it follows that \ref{firstnontr} must also represent a nonzero element in $\redH_{i_{1}+i_{2}-k-1}(\overline{\LLL},\check{\LLL})$.
    Thus there must be an order chain $\chi \in X_{i_1-1}^\gamma$ such that $\hat{1}$ appears in the summand 
    $\rho_{i_1-1}(\chi) \shuffle \zeta_\chi$ of \ref{firstnontr}. 
    
    By $\rho_{i_1-1}(\chi) \in \Sy_{i_1-1}(\overline{\LLL})$ of \ref{ro} we know that the maximal elements of order chains appearing in $\rho_{i_1-1}(\chi)$ are $(i_{1}-1)$-synors. Analogously, the maximal elements of order chains appearing in $\zeta_\chi$ are $(i_2-k-1)$-synors. Since the computation of the shuffle product $\rho_{i_1-1}(\chi) \shuffle \zeta_\chi$ requires taking joins, there must be an $(i_1-1)$-synor $x$ appearing in $\rho_{i_1-1}(\chi)$ and an $(i_2-k-1)$-synor $y$ appearing in $\zeta_\chi$ such that
    $x \vee y = \hat{1}$.

    We are almost finished: We just need to find an $(i_{2}-1)$-synor $z$ that satisfies $y\leq z$; then, $x\vee z=\hat{1}$ will be the required decomposition by the monotonicity of the join. In the $k=0$ case, we can simply choose $z=y$. Otherwise, we look into $\chi$: the whole of $\chi$ lies above $\zeta_{\chi}$ and thus above $y$. By construction the $s$-th element of $\chi$, for $s=0,\ldots,i_{1}-1$, is an $(i_{1}+i_{2}-k-1-s)$-synor. Thus we can find a $t$-synor in $\chi$ for every $i_{2}-k \leq t\leq i_{1}+i_{2}-k-1$. As $1\leq k\leq i_{1}$, the number $i_{2}-1$ falls in this interval, and hence the corresponding $(i_{2}-1)$-synor $z$ appearing in $\chi$ yields the desired element.
    \end{proof}

    The proof of \ref{mainstart} is now immediate.

    \begin{proof}[Proof of \ref{mainstart}:]
       Apply \ref{mainthm} in the case $k = 0$.
       \end{proof}    

\section{Applications}
\label{sec:further}

Let us give two more applications of \ref{mainthm} to 
the Betti table $\beta_{ij}(S/\mathfrak{I})$ of a monomial ideal
$\mathfrak{I}$.
First we introduce multigraded Betti numbers
$\beta_{i,m}(S/\mathfrak{I}) = \dim_\KK \mathrm{Tor}^S_i(S/\mathfrak{I},\KK)_m$ 
for $m \in \overline{\lcm(\mathfrak{I})}$. 
Recall that by \cite{GPW} we have
$\beta_{i,m} (S/\mathfrak{I}) = \dim_\KK \redH_{i-2}\big(\,(1,m)\,)$.
We set $$a_i(S/\mathfrak{I}) = \big| \{ m ~|~\beta_{i,m}(S/\mathfrak{I}) \neq 0 \} \big|.$$

\begin{Theorem} \label{thm:ksubad}
    Let $\mathfrak{I}$ be a monomial ideal in $S$ and $i_1,i_2,k$ integers such that $i_1,i_2 \geq 1$ and $0 \leq k \leq \min\{i_1,i_2\}$. Let $m \in \lcm(\mathfrak{I})$ be such that 
    $$\beta_{i_1+i_2-k,m}(S/\mathfrak{I}) = \dim_\KK \redH_{i_1+i_2-k-2} \big( \, (1,m)\,\big) \neq 0.$$
    Then, there exist $n_1,n_2 \leq m$ in $\lcm(\mathfrak{I})$ such that $m = \lcm(n_1,n_2)$ with $\beta_{i_1,n_1}(S/\mathfrak{I})
    \neq 0$ and $\beta_{i_2,n_2}(S/\mathfrak{I}) \neq 0$.

    In particular, we have 

    \begin{itemize}
        \item[(i)] for all $i_1,i_2 \geq 0$ and $0 \leq k \leq \min\{i_1,i_2\}$,
    $$t_{i_1 + i_2 - k}(S/\mathfrak{I}) \leq t_{i_1}(S/\mathfrak{I})+ t_{i_2}(S/\mathfrak{I}),$$
       \item[(ii)] and for all $i_1,i_2 \geq 0$, $$a_{i_1 + i_2}(S/\mathfrak{I})  \leq a_{i_1}(S/\mathfrak{I}) \cdot  a_{i_2}(S/\mathfrak{I}).$$
       \end{itemize}
    
\end{Theorem}
\begin{proof}
    By assumption, $m \in \lcm(\mathfrak{I})$ is an $(i_1+i_2-k-1)$-synor. Apply \ref{mainthm} to $\LLL = [1,m]$ 
    to obtain an $(i_1-1)$-synor $n_1$ and an $(i_2-1)$-synor $n_2$ in $\LLL$ such that $m = n_1 \vee n_2 = \lcm(n_1,n_2)$. Thus
    $\beta_{i_1,n_1}(S/\mathfrak{I}) = \dim_\KK \redH_{i_1-2}\big(\,(1,n_1) \,) \neq 0$ and $\beta_{i_2,n_2}(S/\mathfrak{I}) = \dim_\KK \redH_{i_2-2}\big(\,(1,n_2) \,) \neq 0$. This concludes the first part of the proof. 

    It remains to verify (i) and (ii). If $i_1=0$ or $i_2=0$, then both statements are trivial. Assume $1 \leq i_1,i_2$ and $t_{i_1 + i_2 - k}(S/\mathfrak{I}) > 0$.
    For assertion (i), choose an $m$ realizing the maximal shift in homological degree $i_1 + i_2 - k$, so that $\deg (m) = t_{i_1 + i_2 - k}(S/\mathfrak{I})$, and observe that $\deg(n_1) \leq t_{i_1}(S/\mathfrak{I})$ and $\deg(n_2) \leq t_{i_2}(S/\mathfrak{I})$.
    This concludes the proof of (i) by 
    \begin{align*} 
      t_{i_1 + i_2 - k}(S/\mathfrak{I}) & = \deg(m) = \deg(\lcm(n_1,n_2)) 
    \leq \deg(n_1) + \deg(n_2) \\ & \leq t_{i_1}(S/\mathfrak{I})+t_{i_2}(S/\mathfrak{I}).
    \end{align*}

    For assertion (ii) set $k = 0$. Then observe that a pair
    $(n_1,n_2)$ for which $\beta_{i_1,n_1}(S/\mathfrak{I}) \neq 0$
    and
    $\beta_{i_2,n_2}(S/\mathfrak{I}) \neq 0$ can only contribute
    the single $m = \lcm(n_1,n_2)$ to the count of $a_{i_1+i_2}(S/\mathfrak{I})$. 
    Thus we obtain the desired inequality $ a_{i_1+i_2}(S/\mathfrak{I})\leq a_{i_1}(S/\mathfrak{I})\cdot  a_{i_2}(S/\mathfrak{I})$.

\end{proof}

Now  \ref{thm:subad} becomes a special case of  \ref{thm:ksubad}(i).

\begin{proof}[Proof of \ref{thm:subad}:]
Apply \ref{thm:ksubad}(i) in the case $k = 0$. 
\end{proof}

We note that subadditivity is a sharp inequality. For example, for a natural number $a \geq 1$ the ideals 
$\mathfrak{I}_a = \big(\,x_1^{a}, \ldots, x_n^{a}\,\big)$ in $S$ 
satisfy $t_k(S/\mathfrak{I}_a) = ak$ for all $0 \leq k \leq \text{pd}(S/\mathfrak{I}_a)$.

On the other hand the following example shows that \ref{thm:ksubad}(i) is
indeed stronger than \ref{thm:subad}.

\begin{Example}
   Consider the ideal
   $I= \big(af,bf,cf,df,ef,abcde)$ in the polynomial ring
   $\KK[a,b,c,d,e,f]$. Using Macaulay2, we get the following Betti table (which can also be easily verified using \eqref{eq:lcm})

   \bigskip
   
   \begin{center}
   \begin{tabular}{rcccccc}
            & 0 & 1 & 2 & 3 & 4 & 5 \\
o5 = total: & 1 & 6 & 11 & 10 & 5 & 1 \\
         0: & 1 & . & . & . & . & . \\
         1: & . & 5 & 10 & 10 & 5 & 1 \\
         2: & . & . & .  &. & . & . \\
         3: & . & . & .  &. &. & .\\
         4: & . & 1 & 1  & . & . & .
         \end{tabular}
         \end{center}

         \bigskip
         
   from which we obtain
   $t_{0}=0$, $t_{1}=5$, $t_{2}=6$, $t_{3}=4$, $t_{4}=5$, $t_{5}=6$.
   \ref{thm:subad} for $i_1+i_2 = 4$ yields
   \begin{align*}
      5=t_{4} & \leq \min\{\,t_{3}+t_{1},t_{2}+t_{2}\,\}=\min\{\,9,12\,\}=9
   \end{align*}
   whereas by \ref{thm:ksubad}(i) for $i_1 = i_2 = 3$ and $k=2$ 
   we get
   \begin{align*}
      5=t_{4} & \leq t_{3}+t_{3}=8.
   \end{align*}
  Let us take a look into a class of ideals that extends the above example. We start by gluing a hollow $p$-simplex to a hollow $q$-simplex along a vertex, where $p> q\geq 2$. The resulting simplicial complex $K_{p,q}$ yields a face lattice which is of course atomic. By \cite{Phan}, there is a corresponding monomial ideal $I_{p,q}$ whose LCM-lattice coincides with the face lattice of $K_{p,q}$. Carrying out the correspondence, it turns out that $I_{p,q}$ can be taken to be an ideal in $K[x_{0},x_{1},\ldots,x_{p},y_{0},y_{1},\ldots,y_{q}]$ and generated by the following families of monomials: 
  \begin{itemize}
      \item $x_{0}y_{0}$
      \item $Y\cdot x_{i}$, for $i=1,\ldots,p$
      \item $X\cdot y_{j}$, for $j=1,\ldots,q$
  \end{itemize} where $X=\prod_{i=0}^{p}x_{i}$ and $Y=\prod_{j=0}^{q}y_{j}$.

The sequence of maximal shifts of $I_{p,q}$ is then $t_{0}=0,t_{1}=p+2,t_{2}=p+3,\ldots,t_{q+1}=p+q+2,t_{q+2}=2q+3,t_{q+3}=2q+4,\ldots t_{p+1}=p+q+2$ where the computation is an easy application of \ref{eq:lcm}. One should note that the two maxima in the sequence correspond exactly to the dimensions of the holes of the two simplices. 

By \ref{thm:ksubad}, we get the inequality $t_{q+3} \leq t_{q + 2} + t_{q+2} = 4q + 6$, while the best (nontrivial) inequality attained by \ref{thm:subad} is $t_{q + 3} \leq p + 2q + 5$. Hence, by fixing the value of $q$ and increasing $p$, we see that generalized subadditivity gives a constant bound, while the $k = 0$ case gives a bound growing linearly in $p$.
\end{Example}
Next we show that a synor complex of $\overline{\lcm(\mathfrak{I})}$ 
can be used to construct a minimal free
resolution of $S/\mathfrak{I}$.

Let $\lcm(\mathfrak{I})$ be the LCM-lattice of the monomial ideal $\mathfrak{I}$ in the polynomial ring $S$ and $\big(\,\Sy_*(\,\overline{\lcm(\mathfrak{I})}\,),\partial_*\,\big)$ 
a synor complex for $\overline{\lcm(\mathfrak{I})} = \lcm(\mathfrak{I}) \setminus \{ \hat{0}\}
= \lcm(\mathfrak{I}) \setminus \{ 1 \}$ where as usual
we consider $1$ as the monomial $1 = x_1^ 0\cdots x_n^ 0$.
For a number $i \geq 0$ and a monomial $m \in \overline{\lcm(\mathfrak{I})}$ set 
$$F_i = \bigoplus_{m \in \overline{\lcm(\mathfrak{I})}} \Sy_{i-1}\big(\,\overline{\lcm(\mathfrak{I})}\,\big)^{(m)} \otimes S(-m).$$ Here
we consider $\Sy_{-1}\big(\,\overline{\lcm(\mathfrak{I})}\,\big)$ as concentrated in multidegree 
$1 = x_1^ 0\cdots x_n^ 0$ and write $S(-m)$ for the multigraded 
rank $1$ free $S$-module where 
the multidegree of a monomial $m'$ is $mm'$. 
For $i \geq 1$ we define $\delta_i : F_i \rightarrow F_{i-1}$ in the 
following way.
Let $\zeta \in \Sy_{i-1}\big(\,\overline{\lcm(\mathfrak{I})}\,\big)^{(m)}$ be a synor chain.  
Since a synor complex of $\overline{\lcm(\mathfrak{I})}$ is strictly
$\overline{\lcm(\mathfrak{I})}$-graded we have that
$$\partial_{i-1}(\zeta) = \sum_{m' \in \overline{\lcm(\mathfrak{I})_{< m}}} \zeta_{m'}$$  
for chains $\zeta_{m'} \in 
\Sy_{i-2}\big(\,\overline{\lcm(\mathfrak{I})}\,\big)^{(m')}$. We then set
$$\delta_i(\zeta \otimes 1) = \sum_{m' \in \overline{\lcm(\mathfrak{I})}_{< m} }\frac{m}{m'} \,\zeta_{m'} \otimes 1.$$
Note the fact that $\big(\,\Sy_*\big(\,\overline{\lcm(\mathfrak{I})}\,\big),\partial_*\,\big)$ is strictly $\overline{\lcm(\mathfrak{I})}$-graded
implies that $\delta_*$ is a well defined homomorphism of
multigraded free $S$-modules.

We call $(F_*,\delta_*)$ a \defn{synor-resolution} of $S/\mathfrak{I}$.

\begin{Theorem} \label{thm:synorres}
    A synor-resolution of a monomial ideal $\mathfrak{\mathfrak{I}}$ is a
    minimal free resolution of $S/\mathfrak{I}$.
\end{Theorem}
\begin{proof}
    First we must show that the synor-resolution $(F_*,\delta_*)$  is a resolution of $S/\mathfrak{I}$.
    
    We have seen that $(F_*,\delta_*)$ is a sequence of 
    homomorphisms of multigraded free $S$-modules. To prove that it
    indeed is an exact complex of multigraded modules it suffices
    to verify for each monomial, or equivalently multidegree, $m$, 
    that the $m$-graded part  $(F_*^ {(m)},\delta_*|_{F_*^ {(m)}})$ of 
    $(F_*,\delta_*)$ is an
    exact sequence of 
    homomorphisms of vector spaces. 
    From $F_i^{(m)} \cong \displaystyle{\bigoplus_{m' \in \overline{\lcm(\mathfrak{I})}_{\leq m}}} \Sy_{i-1}(\overline{\lcm(\mathfrak{I})})^{(m')}$ 
    and the construction of $\delta_*$ we deduce that $(F_*^ {(m)},\delta)$
    is an exact
    sequence if and only if for $\JJJ= \overline{\lcm(\mathfrak{I})}_{\leq m}$ we have that
    $(\Sy_*(\LLL)^\JJJ,\partial_*)$ is exact.
    The latter is satisfied by (S2) and the fact that 
    $\redH_*(\JJJ) = 0$ since $\JJJ$ is a cone over $m$. 

    It remains to verify that the cokernel of $\delta_1$ is $S/\mathfrak{I}$. 
    Let $m_1,\ldots, m_r$ be the minimal monomial generators 
    of $\mathfrak{I}$.  These then are also the
    minimal elements of $\overline{\lcm(\mathfrak{I})}$. It follows that 
    $\Sy_0(\LLL)^{(m)} = \KK$ if $m = m_i$ for some $i$
    and $0$ otherwise. For 
    $1 \in \Sy_0(\LLL)^{(m)}$ we have  $\partial (1) = 1$ and
    hence $\delta_1(1) = m (1 \otimes 1)$. It follows
    that the image of $\delta_1$ in 
    $F_{0} = \Sy_{-1}(\LLL) \otimes S(-1)
    \cong S$ is $\mathfrak{I}$. This completes the proof.
\end{proof}
The preceding proof uses arguments very similar to the proof of Proposition 1.2 in \cite{BS}. On the other hand the setting here is purely homological and there does not have to be a cellular complex whose cellular chain complex is
the synor complex.

\end{document}